\documentclass{article}

\pdfoutput=1
\usepackage{graphicx}
\usepackage{amsmath, color, soul}
\usepackage{amsfonts}
\usepackage{amscd}
\usepackage{epic}
\usepackage{eepic}
\usepackage{color}

\def\R{\mathbb{R}}
\def\BbR{\mathbb{R}}
\def\BbN{\mathbb{N}}

\def\BbZ{\mathbb{Z}}

\def\e{\varepsilon}
\def\a{\alpha}
\def\b{\beta}

\def\vp{\varphi}

\def\E{\mathcal{E}}
\def\D{\mathcal{D}}
\def\F{\mathcal{F}}

\def\SS{\Sigma}
\def\bE{{\mathbb E}}
\def\bP{{\mathbb P}}
\def\uJ{\underline{J}}

\def\CSJ{{\rm CSJ}}

  \def\sd#1#2{#1\setminus #2}
 
\usepackage{amsthm}

\theoremstyle{plain}
\newtheorem{Th}{Theorem}[section]
\newtheorem{Lem}[Th]{Lemma}
\newtheorem{Cor}[Th]{Corollary}
\newtheorem{Prop}[Th]{Proposition}
\newtheorem{Rem}[Th]{Remark}
\newtheorem{rem}[Th]{Remark}

\theoremstyle{definition}
\newtheorem{Def}[Th]{Definition}
\newtheorem*{Not}{Notation}
\newtheorem{Ex}[Th]{Example}
\newtheorem{Assum}[Th]{Assumption}

\theoremstyle{remark}
 
\def\thm{\begin{Th}}
\def\endthm{\end{Th}}
\def\lemma{\begin{Lem}}
\def\endlemma{\end{Lem}}
\def\cor{\begin{Cor}}
\def\endcor{\end{Cor}}
\def\prop{\begin{Prop}}
\def\endprop{\end{Prop}}
\def\definition{\begin{Def}}
\def\enddefinition{\end{Def}}
\def\remark{\begin{Rem}}
\def\endremark{\end{Rem}}
\def\example{\begin{Ex}}
\def\endexample{\end{Ex}}
\def\demo{\begin{proof}}
\def\enddemo{\end{proof}}

\def\notation{\begin{Not}}
\def\endnotation{\end{Not}}
\def\assumption{\begin{Assum}}
\def\endassumption{\end{Assum}}

\def\sX{{\mathcal X}}
\def\sE{{\mathcal E}}
\def\sF{{\mathcal F}}
\def\sN{{\mathcal N}}

\def\bR{{\mathbb R}}

\def\bP{\mathbb P}
\def\bE{\mathbb E}

\def\diam{{\rm diam}}

\def\<{\langle}
\def\>{\rangle}
\def\1{1\!\!\!\!1}

\def\eps{\varepsilon}

\def\be{\begin{equation}}
\def\ee{\end{equation}}

\begin{document}

\title{\bf  Heat kernel lower bound estimates for  symmetric pure jump processes  via averaged jump kernels }
 
\author{{\bf Zhen-Qing Chen}  and {\bf Jun Kigami} } 
 
\date{June 22, 2026}
 
\maketitle 

\begin{abstract}  
We derive a heat kernel lower bound  estimate for symmetric  pure jump processes on general volume doubling metric measure spaces with possible degenerate  and/or singular jump kernels using averaged jump kernels. As an application, the main result of this paper is applied to derive a lower bound estimate for the transition density function of the trace of Brownian motions on Sierpinski gaskets on the bottom of the Sierpinski gasket.

\bigskip


 
\end{abstract} 

\section{Introduction}\label{INT}

In this paper, a jump process refers to a strong Markov process that has discontinuous sample paths and does not have diffusive component. The infinitesimal generator of a jump process is a non-local operator that does not have a differential operator component. A L\'evy process without Gaussian component is an example of a jump process in  Euclidean spaces. In general, the jump measure of a jump process can be state-dependent. In recent years, there are lots of activities in the study of jump processes both in probability theory and in analysis, from the theoretic as well as applied perspectives. An important direction of these studies is the heat kernel analysis for jump processes and non-local operators. \par

 Suppose that $(\sX, d)$ is locally compact separable metric space and $m$ is Radon measure on $\sX$ with full support. For $x\in \sX$ and $r>0$, denote by $B(x, r)$ the open ball centered at $x$ with radius $r>0$. Consider a  symmetric regular Dirichlet form $(\E, \F)$ on $L^2(\sX; m)$ of the following pure jump type:
\[
\E(u, v) = \int_{\sX \times \sX}  (u(x) - u(y))(v(x) - v(y)) J(x, y) m(dx)m(dy)
\quad \hbox{for } u, v \in \F.
\]
Here $J(x, y)$ is a non-negative symmetric kernel defined on $\sX \times \sX \setminus {\rm diag}$, where ${\rm diag}:=\{(x, x): x\in \sX\}$ denotes the diagonal set of the product space $\sX \times \sX$. Denote by $X$ the symmetric Hunt process associated with  $(\E, \F)$ on $L^2(\sX; m)$.\par

Suppose that there is an increasing function $V(r)$ on $[0, \infty)$ that satisfies doubling and reverse doubling property so that 
\begin{equation} \label{e:1.1}
m(B(x, r)) \asymp V(r) \quad \hbox{for all } x\in \sX\,\, \hbox{and}\,\, 0 < r< {\rm diam}(\sX).
\end{equation}
It is shown in \cite{CK03, CK08} that for $0<\alpha <2$, if 
\begin{equation} \label{eI:1.2}
J(x, y) \asymp \frac{1}{V(d(x, y))d(x, y)^\alpha} \quad \hbox{on } \sX \times \sX \setminus {\rm diag},
\end{equation} 
then the symmetric Hunt process $X$ has a jointly continuous transition density function $p(t, x, y)$ with respect to $m$ and that
\begin{equation} \label{e:1.3}
p(t, x, y)  \asymp \frac1{V(t^{1/\alpha})} \wedge \frac{t}{V(d(x, y))d(x, y)^\alpha} \quad \hbox{on }  (0, {\rm diam}(\sX)) \times \sX \times \sX.
\end{equation}

Conversely, if the symmetric Hunt process $X$ has a jointly continuous transition density function $p(t, x, y)$ with respect to $m$  that satisfies the two-sided estimates \eqref{eI:1.2}, then \eqref{e:1.1} holds. Here for $a, b\in \R$, $a \wedge b :=\min \{ a, b \}$. The notation $\asymp$ is as explained at the end of this introduction. As a particular case, when $(\sX, d, m)$ is Ahlfors $\eta$-regular, that is, when $V(r)=r^\eta$ in \eqref{e:1.1}, we have 
 $$
 J(x, y) \asymp \frac{1}{d(x, y)^{\eta + \alpha}} \quad \hbox{on } \sX \times \sX \setminus {\rm diag}
 $$
 if and only if the symmetric Hunt process $X$ has a jointly continuous transition density function $p(t, x, y)$ with respect to $m$ satisfying
\begin{equation} \label{e:1.4} 
 p(t, x, y)  \asymp t^{-\eta/ \alpha} \wedge \frac{t}{d(x, y)^{\eta+ \alpha}}  
 \quad \hbox{on }  (0, {\rm diam}(\sX)^\alpha ) \times \sX \times \sX.
  \end{equation}
 In this case, the heat kernel estimates can be restated as 
 \begin{equation} \label{e:1.5} 
 p(t, x, y)  \asymp t^{-\eta/ \alpha} \wedge tJ(x, y)
 \quad \hbox{on }  (0, {\rm diam}(\sX)^\alpha ) \times \sX \times \sX.
 \end{equation}
Such a symmetric Hunt process $X$  is called symmetric $\alpha$-stable-like process on $\sX$ in \cite{CK03}.
 The above results have been further extended in \cite{CKW1} allowing general volume doubling measure instead of uniform volume comparability condition \eqref{e:1.1} as well as more general scale function $\phi (r)$ (alas mixed stable-like scaling) instead of the $\alpha$-stable scaling
 function  $r^\alpha$. Moreover, various characterizations for the two-sided heat kernel estimates and for the upper bound heat kernel estimates are obtained in \cite{CKW1} for symmetric jump diffusions on metric measure spaces. See also \cite{GHH, MSC} for related works. In all these work on the two-sided heat kernel estimates restricted to metric measure spaces $(\sX, d, m)$ that satisfy uniform volume comparability condition \eqref{e:1.1}, the jump kernel $J(x, y)$ has the property that it is comparable to $j_0(d(x, y))$ for some decreasing function $j_0$ on $(0, \infty)$ so, in particular, the jump kernel $J(x, y)$ is non-degenerate. However, there are many natural circumstance in which the symmetric jump kernel does not satisfy such a condition. While there are some results on getting (rough) heat kernel estimates for such kind of symmetric pure jump processes (see, e.g., \cite{CKW1}),
so far there is no systematic study on getting heat kernel lower bound estimates for these processes.
The purpose of this paper is to fill this gap.   More precisely, in Theorem~\ref{GT.thm10}, we will give an off-diagonal lower heat kernel estimate in terms of the averaged jump kernel $J(x, y, r)$ define by
\[
J(x, y, r) = \frac 1{m(B(x, r))m(B(y, r))}\int_{B(x, r)}\int_{B(y, r)} J(u, v)m(du)m(dv).
\]
See \eqref{INT.eq100} below for more details.

\par
 Before stating (a special case of) the main result of this paper, we present an example of 
 a symmetric jump process whose jump kernel $J(x, y)$ has the property that 
 \begin{equation}\label{e:1.6}
 0<\liminf_{|x-y|\to 0} J(x, y) <\limsup_{|x-y|\to 0} J(x, y)=\infty.
 \end{equation}
 (See Example \ref{E:3.3} below for an  example where the jump kernel $J(x, y)$ vanishes on some regions of $\R^d\times \R^d$.) Our example here is the jump kernel $J_*$ of the trace of the Brownian motion on the Sierpinski gasket $K$ on its bottom line $I$, which can be naturally identified with $[0, 1]$. See Figure~\ref{Fig00} for the Sierpinski gasket $K$  and its bottom line $I$. 

\begin{figure}\label{Fig00}
\centering
\centering
\includegraphics[width=\linewidth]{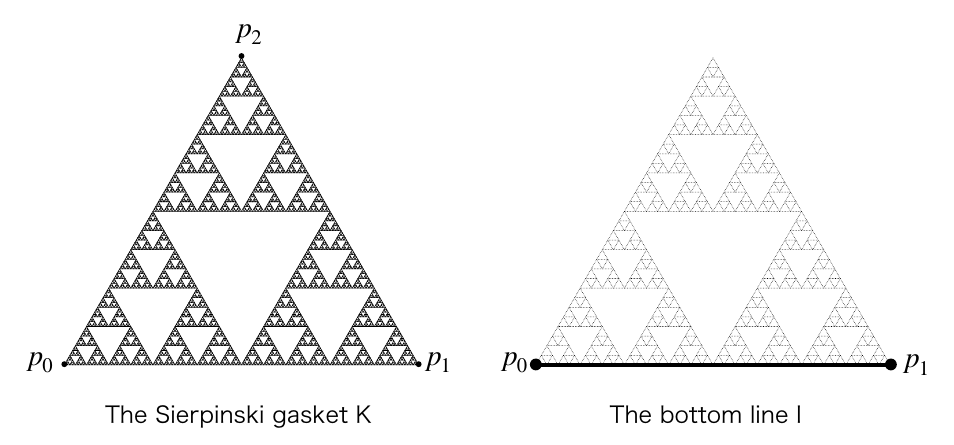}
\vspace{-20pt}
\caption{The Sierpinski gasket and its bottom line $I = \overline{p_0p_1}$}\label{Fig0}
\end{figure}

The general theory of the Beurling-Deny decomposition of trace processes was developed in \cite{CFY, CF2}.
It is well know that an isotropic $\alpha$-stable process on $\R^n$ for $\a \in (0, 2)$ can be viewed as trace process of some reflected diffusion process in the upper half space ${\mathbb H}^{n+1}$ on its boundary; see \cite{MO}. 
In particular, the trace of the reflected Brownian motion ${\mathbb H}^{n+1}$ on $\partial {\mathbb H}^{n+1}$ is the isotropic Cauchy process 
on $\R^n$.

Our trace process $X$ on $I$ is a conservative symmetric pure jump process. Moreover, its jump measure $J_*(dx, dy)$
has density $J_*(x, y)$ with respect to $\nu(dx)\nu(dy)$, where $\nu$ is the Lebesgue measure on $I$, on $I\times I \setminus {\rm diag}$ (see \cite[(1.10)-(1.12)]{CC}), as the harmonic measure on $I$ of the Brownian motion  starting from the top vertex $p_2$ on the Sierpinski gasket $K$ coincides with $\nu$.  In fact, an exact formula for $J_*(x, y)$ has recently been derived in \cite{KiKTaka} that has the singular behavior  
\eqref{e:1.6};  see Theorem~\ref{TSS.thm10} below. 
From it, one can see that  
\begin{equation} \label{e:J}
J_*(x, y) \leq \frac{c_*}{ |x-y|^{\a + 1}} \quad \hbox{for } x, y \in I
\end{equation} 
 where  $\a = \frac{\log (10/3)}{\log 2} \in(1, 2)$; see \eqref{e:3.33}.
 But the comparable lower bound does not hold for $J_*$. \par
 It is not difficult to see from the known results (see Theorem \ref{TSS.thm20} below) that the following upper bound estimate holds for the  transition density $p_* (t, x, y)$,   which is denoted by $p^I_{\nu}(t, x, y)$ in Section~\ref{TSS}, of the trace process $X$ on $I$ with respect to the Lebesgue measure $\nu$:
there exists $c_1 > 0$ such that 
\begin{equation}\label{e:1.7}
p_* (t, x, y) \le c_1\min\bigg\{t^{-1/\alpha}, \  \frac{t}{|x - y|^{1+\alpha}}\bigg\}
\end{equation}
for any $x, y \in I$ and $t \in (0, 1]$. This upper estimate is called upper heat kernel kernel estimate, {\rm UHK} for short. Moreover, there are positive  constants
$c_2$ and $c_3 $ such that
\begin{equation}\label{e:1.8}
  p_* (t, x, y) \geq c_2 t^{-1/\alpha} 
\end{equation}
for any  $t\in (0, 1]$ and $x, y \in I$ with $|x - y| \le \e_1t^{1/\alpha} $; see Theorem \ref{T:3.1}. This estimate is called neqr-diagonal lower heat kernel estimate, {\rm NLHK} for short.\par
However, in contrast to the situation for \eqref{e:1.4},  the
comparable off-diagonal lower bound estimate can not hold for $p_* (t, x, y)$ as we know that 
\begin{equation}\label{e:I.9}
\lim_{t\to 0} \frac{p_* (t, x, y)} {t} =  J_*(x, y) \quad  \hbox{ weakly on } I \times I \setminus {\rm diag}.
\end{equation} 
In view of \eqref{e:J}, \eqref{e:1.7} and \eqref{e:1.8}, as an application of the main result, Theorem \ref{GT.thm10}, of this paper, we are able to derive a lower bound off-diagonal estimate for the transition density $p_* (t, x, y)$ of the trace process $X$ on $I$ in Theorem~\ref{TSS.thm40} as follows:\,\, There exists constants $C > 0$ and $\e_2 \in (0, \e_1/2)$ such that, for any $t \in (0, 1]$ and $x, y \in I$ with $|x - y| > \e_1t^{1/\a}$,
\begin{equation}\label{e:1.10}
p_* (t, x, y) \ge CtJ_*(x, y, \e_2t^{1/\a}),
\end{equation}
where $J_*(x, y, r)$ is the averaged jump kernel defined by
\[
J_*(x, y, r) = \frac 1{\nu(B(x, r))\nu(B(y, r))}\int_{B(x, r)}\int_{B(y, r)} J(u, v)\nu(du)\nu(dv).
\]
Thanks to the following UJS condition for $J_*$, which is proven in Proposition~\ref{TSS.prop10}, i.e. 
 \begin{equation}\label{e:UJS}
 J_* (x, y) \leq \frac{2}{\nu (B(x, r))} \int_{B(x, r)} J_*(z, y) \nu (dz) \quad \hbox{for any } 0<r<|x-y|/2,
 \end{equation} 
for any $x \neq y\in I$ and $r > 0$ with $0 < r < |x - y|/2$, we may replace the averaged jump kernel $J_*(x, y, \e_2t^{1/\a})$ with the jump kernel $J(x, y)$ and obtain a counter part of \eqref{e:1.5}:
\begin{equation}\label{INT.eq30}
p_* (t, x, y) \ge c_4 tJ_* (x, y)
\end{equation}
for any $t \in (0, 1]$ and $x, y \in I$ with with $|x - y| > \e_1t^{1/\a}$.
 Although this looks neater than \eqref{e:1.10}, it is less sharper because $J_*(x, y, \e_2t^{1/\a}) \ge cJ_*(x, y)$ and the averaged jump kernel $J_*(x, y, r)$ is continuous on 
 $$
 \{(x, y, r): \  x, y \in I, x \neq y, 0 < r < |x - y|/2\}
 $$
  while $J_*(x, y)$ is highly discontinuous over $I^2$.

\par

In general setting of a pure jump process with a jump kernel $J$ on a metric measure space $(\sX, d, m)$, 
our main theorem, Theorem~\ref{GT.thm10}, states that under the conservativeness of  the jump process and generalized versions of UHK \eqref{e:1.7} and NLHK \eqref{e:1.8} (see Definition \ref{S2.def10} below), we have the following off-diagonal lower hear kernel estimate; there exist constants $C > 0$, $\kappa \in (0, 1)$ and $\e_1 \in (0, \e_2)$, where $e_2$ is the constant appearing in NLHK, such that
\begin{equation}\label{INT.eq100}
p(t, x, y) \ge CtJ(x, y, \e_2\phi^{-1}(t))
\end{equation}
for any $t \in (0, \kappa\diam{\sX})$ and $x, y \in \sX$ with $d(x, y) \ge \e_1\phi^{-1}(t)$. Note that  the lower bound estimate in \eqref{INT.eq100}  is asymptotically sharp in the sense of  \eqref{e:I.9} as by Lebesgue
differentiation theorem, 
$$
\lim_{t\to 0} \frac{t J(x, y, \e_2\phi^{-1}(t))}{t} =\lim_{t\to 0}  J(x, y, \e_2\phi^{-1}(t)) = J(x, y)
\quad \hbox{for a.e. } x\not= y \hbox{ in } \sX.
$$

\medskip

The rest of this paper is organized as follows. 
In Section \ref{S:2}, we develop a general theory for the off-diagonal lower bound estimate 
\eqref{e:1.10} for symmetric pure jump processes on metric measure spaces. Some  results and  techniques 
developed in \cite{CKW1, CKW2} in the study of heat kernels for pure non-local Dirichlet forms 
are utilized.  The main result of this paper, Theorem \ref{GT.thm10} is established there. Section \ref{S:3} is devoted to the Sierpinski gasket example mentioned above. The jump kernel $J_*$ for the trace of Brownian motion on Sierpinski gasket is modified in Section \ref{S:4} so that it has an equivalent resistant form on the unit interval $I$ for which we can still get the lower bound off-diagonal estimate. In Section \ref{EIC}, an infinite Sierpinski gasket example is presented to illustrate the main result of this paper. 

\medskip

In this paper, we use $:=$ as a way of definition. For $a, b\in \R$, $a\wedge b:=\min\{a, b\}$ and $a\vee b :=\max \{a, b\}$. 
For  a set $A$ and   two non-negative real-valued functions $f$ and $g$ on $A$, we write
\[
f(x) \asymp g(x)
\]
for any $x \in A$ if there exist positive constants $c_1$ and $c_2$ such that
\[
c_1f(x) \le g(x) \le c_2f(x)
\]
for any $x \in A$. Moreover, for a metric space $(X, d)$, we denote the collection of real-valued continuous functions by $C(X, d)$ or simply $C(X)$ if the metric is obvious from a context.

\setcounter{equation}{0}
\section{Basic framework and general theory}\label{S:2}

In this section, we give our basic framework to study lower bound estimate of heat kernel associated with pure jump processes. In addition, we will present some relations between properties regarding Dirichlet forms and heat kernel estimates like the Poincar{\'e} inequality, cut-off Sobolev inequality and so on.\par
 
Let $(\sX, d)$ be a locally compact separable metric space, 
and $m$ a positive Radon measure on $\sX$ with full support.
We will refer to such a triple $(\sX, d, m)$ as a \emph{metric measure space}. 
Throughout this paper, we assume that $(\sX, d, m)$
satisfies a   volume doubling  (VD) and reversed volume doubling (RVD) property; that is, 
there are constants $c_{\ref{e:3}.2} > c_{\ref{e:3}.1} > 1$ and $L>0$ so that
\begin{equation}\label{e:3}
 c_{\ref{e:3}. 1}V(x, r) \leq V(x, Lr)\leq c_{\ref{e:3}. 2}V(x, r)  
 \quad \hbox{for all } x\in \sX \hbox{ and } r>0,
\end{equation}
 where $V(x, r):= m (B(x, r))$. Condition \eqref{e:3} is equivalent the existence of  constants $c_{\ref{e:1.2a}} \geq 1$ and $\beta_{\ref{e:1.2a}. 2} \geq \beta_{\ref{e:1.2a}. 1} > 0$ so that
\begin{equation}\label{e:1.2a}
 c_{\ref{e:1.2a}}^{-1} (R/r)^{\beta_{\ref{e:1.2a}. 1}} \leq V(x, R)/V(x, r) \leq c_{\ref{e:1.2a}}(R/r)^{\beta_{\ref{e:1.2a}. 2}}
\end{equation} 
 for all  $ x\in \sX$ and $0<r\leq R< {\rm diam}(\sX)$.\par
 
Let $(\E, \F)$ be a pure jump regular symmetric Dirichlet form on $L^2(\sX; m)$; that is, $(\E, \F)$ admits the following expression:
\begin{equation}\label{GT.eq100}
 \E(u, v) = \frac12 \int_{\sX \times \sX} (u(x) - u(y))  (v(x) - v(y))  J(dx, dy)
 \end{equation}
for any $u, v \in \F$,  where $J(dx, dy)$ is a symmetric Radon measure on $\sX \times \sX \setminus {\rm diag}$, and $\F \cap C_c(\sX))$ is dense both in $(C_c(\sX), \| \cdot \|_\infty)$ and in the Hilbert space $(\F, \E_1)$. Here $\| f\|_\infty = \sup_{x\in \sX} |f(x)|$ for $f\in C(\sX)$ and $\E_1 (u, v)= \E(u, v) + \int_{\sX} u(x) v(x) m(dx)$. In this paper, we assume 
\begin{equation}\label{GT.eq105}
J(dx, dy) = J(x, y)m(dx)m(dy)
\end{equation}
for  some non-negative symmetric measurable function $J(x, y)$ on $\sX \times \sX \setminus {\rm diag}$.

\definition\label{GT.def00}
Let $(\sX, d)$ be a locally compact separable metric space, $m$ be a positive Radon measure on $\sX$ with full support and $(\E, \F)$ be a regular Dirichlet form on $L^2(\sX, m)$. We call $(\sX,d,m,\sE, \sF)$ a measure metric Dirichpet space (MMD space in abbreviation) if the conditions \eqref{e:3}, \eqref{e:1.2a}, \eqref{GT.eq100} and \eqref{GT.eq105} are satisfied.
\enddefinition

It is well known that there is a Hunt process $X$ 
associated with the regular Dirichlet form $(\sE, \sF)$ on $L^2(\sX; m)$,
which is unique in distribution up to $\sE$-quasi-equivalence; see \cite{CF, FOT}. Denote by $\sN$ the properly exceptional set of $X$ so that the Hunt process $X = \{X_t, t\geq 0; \bP_x, x\in \sX \setminus \sN\}$ is well defined for every starting point $x\in \sX \setminus \sN$. We always represent functions in $\sF$ by its $\sE$-quasi-continuous version. For these and related basic facts about regular Dirichlet forms, we refer the reader to \cite{CF, FOT}. 
 
 \medskip
 
Let $\sX_\partial:= \sX \cup \{\partial \}$ be the one-point compactification of $\sX$. 
 Denote by $\{\sF_t\}_{ t\geq 0}$ the minimum augmented filtration generated by $X$. 
 Observe that our Dirichlet form $(\sE, \sF)$ admits no killing measure.
 The jump kernel $J(x, y)$ describes how the associated Hunt process $X$ jumps through the following L\'evy system formula:
 for any non-negative $\{\sF_t\}_{t\geq 0}$-predictable process $Y$ and non-negative Borel measurable function $f$ on $\sX_\partial \times \sX_\partial$
 with $f(x, x)=0$ for $x\in \sX_\partial$
 \begin{equation} \label{e:2.3}
\bE_x \left[ \sum_{s>0} Y_s f(X_{s-}, X_s) \right] = \bE_x \int_0^\infty Y_s  \left( \int_{\sX} f(X_s, y) J(X_s, y) m(dy) \right) ds
 \end{equation}
 for every $x\in \sX \setminus \sN$;  see \cite[(4.3.11) and (A.3.32)]{CF}. \par
  We are going to use the following special case of \eqref{e:2.3} to prove our main theorems in this paper.
 
\lemma\label{GT.lemma100}
Let $\vp: \sX \to \BbR$ be a non-negative Borel measurable function and let $D \subseteq \sX$ be a proper open set. Denote by $\tau_D$ the first exit time from $D$ by $X$, i.e. 
\[
\tau_D = \inf\{t\geq 0:\,X_t\notin D\}.
\]
Then
 \begin{eqnarray} \label{e:exitr} 
&& \bE_x\left[\vp(X_{\tau_D}); \tau_D \le T, X_{\tau_D} \neq X_{\tau_D-}\right]  \nonumber \\
 &&  =\,  \bE_x\left[\int_0^{T \wedge \tau_D}  \int_{D^c} \vp(y)J(X_s, y)m(dy)ds  \right]
 \end{eqnarray}
for any $x \in \sd{D}{\sN}$ and $T \in (0, \infty]$. In particular, if $U$ is a Borel subset of $\sX$ and $U \cap \overline{D} = \emptyset$, then
\begin{equation}\label{GT.eq15}
P_x(X_{T \wedge \tau_D} \in U) = \bE_x\left[\int_0^{T \wedge \tau_D}  \int_{U} J(X_s, y)m(dy)ds  \right].
\end{equation}
\endlemma

\demo 
On the event  $\{X_{\tau_D-}\in \partial D\}$, $\tau_D$ is predictable (as it is the increasing limit of $\tau_{D_n}$ where $\{D_n\}$ is a sequence of relatively compact subset sets of $D$ that increases to $D$) so by the quasi-left-continuity of the Hunt process $X$ we have $X_{\tau_D}= X_{\tau_D-}$. Letting $Y_s = \1_{[0, T \wedge \tau_D]}(s)$ and $f(x, y) = \1_{D}(x)\vp(y)\1_{D^c}(y)$, we have \eqref{e:exitr} from \eqref{e:2.3}. The second equality \eqref{GT.eq15} follows if $\vp = \1_U$.
\enddemo

We say a Hunt process $X$ satisfies the the absolute continuity
condition (AC) if $X$ can be defined to start from every point $x\in \sX$ and that $\bP_x (X_t \in dy) \ll m(dy)$
for every $x\in \sX$ and $t>0$. 
  Under (AC) condition, every $\sE$-polar set is polar and the L\'evy system
formulas \eqref{e:2.3} and \eqref{e:exitr} hold with $\sN = \emptyset$; see 
 \cite[Theorems A.2.17 and Theorem 4.1.11]{CF} and  \cite[Chapter 5]{FOT}

\medskip
 
 Let $\bR_+:=[0,\infty)$, and $\phi: \bR_+\to \bR_+$ be a strictly increasing continuous
function  with $\phi (0)=0$ ,
$\phi(1)=1$
and satisfying that there exist constants $c_{\ref{e:phi}. 1}, c_{\ref{e:phi}. 2} > 0$ and $0 < \beta_{{\ref{e:phi}. 1}} \leq \beta_{{\ref{e:phi}. 2}} <\infty$  such that
\begin{equation}\label{e:phi}
c_{\ref{e:phi}. 1}\Big(\frac Rr\Big)^{\beta_{{\ref{e:phi}. 1}}} \leq
\frac{\phi (R)}{\phi (r)}  \leq c_{\ref{e:phi}. 2}\Big(\frac Rr\Big)^{\beta_{{\ref{e:phi}. 2}}} \quad \hbox{for all } 0 < r \le R. 
\end{equation}
Note that (\ref{e:phi}) is
equivalent to the existence of constants $c_1 \geq c_2>0$ and $l_0>1$ such that
\[
c_1\phi (r) \leq \phi ({l_0}r)\leq c_2\phi (r)\quad \hbox{for all} r > 0.
\]

\medskip

The following definition is taken from \cite{CKW1}.

 \begin{definition}\label{S2.def10}\rm 
 Let $(\sX,d,m,\sE, \sF)$ be an MMD space.
\begin{enumerate}  
\item[\rm (i)] We say that the Dirichlet form $(\sE, \sF)$ is conservative if its associated Hunt process is conservative. 

\item[\rm (ii)] We say that the transition density upper bound estimate ${\rm UHK}(\phi)$ holds for the Dirichlet form $(\sE, \sF)$ if the Hunt process $X$ associated with $(\sE, \sF)$  has a jointly continuous density function
 $p(t, x, y)$ on $(0, \infty )\times \sX \times \sX$ with respect to the measure $m$ and there is a positive constant $c_{\ref{e:2.6}} >0 $ so that 
\begin{equation}\label{e:2.6}
 p(t, x, y) \leq c_{\ref{e:2.6}}\left( \frac{1}{V(x, \phi^{-1}(t))} \wedge
\frac{t}{V(x, d(x,y))\phi (d(x, y))} \right)
\end{equation}
for any $t > 0$ and $x, y \in \sX$.  
Here $\phi^{-1}$ denotes the inverse function of $\phi$.
 
\item[\rm (iii)] We say that the near diagonal lower bound estimate ${\rm NLHK}(\phi)$
 holds for the Dirichlet form $(\sE, \sF)$ if the Hunt process $X$ associated with $(\sE, \sF)$  has a jointly continuous density function $p(t, x, y)$ on $(0, \infty)\times \sX \times \sX$ with respect to the measure $m$ and there exist constants $\e_{\ref{e:1.9}} \in (0,1)$ and $c_{\ref{e:1.9}} > 0$ such that  
\be\label{e:1.9} 
p(t, x ,y )\ge \frac{c_{\ref{e:1.9}}}{V(x , \phi^{-1}(t))} \quad \hbox{for any } x ,y \in \sX
\hbox{ with }  d(x, y) < \e_{\ref{e:1.9}}\phi^{-1}(t)  .
\ee

 \item[\rm (iv)]  We say that the transition density lower bound estimate for the Dirichlet transition density, ${\rm NDL}(\phi)$ for short, holds for the Dirichlet form $(\sE, \sF)$ if the Hunt process $X$ associated with $(\sE, \sF)$   has a jointly continuous density function $p(t, x, y)$ on $(0, \infty)\times \sX \times \sX$ with respect to the measure $m$ and there exist  $\e_{\ref{e:2.15}} \in (0, 1)$ and $c_{\ref{e:2.15}} > 0$ so that for every $x_0 \in \sX$, $r > 0$, $t\in (0, \phi (\e_{\ref{e:2.15}}r))$ and 
 $B = B(x_0, r)$, 
 \begin{equation}\label{e:2.15}
 p^B(t, x, y) \geq   \frac{c_{\ref{e:2.15}}}{V(x_0, \phi^{-1}(t))}  \quad \hbox{for } x, y \in B(x_0, \e_{\ref{e:2.15}} \phi^{-1}(t)).
 \end{equation}
 Here $p^B(t, x, y)$ is the transition density function of the part process $X^B$ of $X$ killed upon leaving the ball $B$. 
\end{enumerate} 
  \end{definition}

  \medskip

\begin{rem}\rm  \begin{enumerate}
\item[(i)] When the Hunt process $X$ has a a jointly continuous density function
 $p(t, x, y)$ on $\sX \times \sX$ satisfying \eqref{e:2.6}, $X$ can be refined to be a Feller process 
on $\sX$ and hence can start from every point in $\sX$. 

\item[(ii)] 
Condition ${\rm NDL}(\phi)$  implies ${\rm NLHK}(\phi)$. Indeed, suppose ${\rm NDL}(\phi)$ holds. 
Then   for any $t>0$ and any $x, y\in \sX$ with $d(x, y) <\eps \phi^{-1} (t)$, 
taking $r=\phi^{-1}(t)/\eps^2$, we get by \eqref{e:2.15} that 
$$
p(t, x, y) \geq p^{B(x, r)} (t, x, y) \geq \frac{c}{V(x , \phi^{-1}(t))}.
$$
That is,  ${\rm NLHK}(\phi)$ holds.   
\end{enumerate}

\end{rem}

 \medskip
  
\lemma\label{L:2.9}
Suppose ${\rm UHK}(\phi)$ and ${\rm NLHK}(\phi)$ hold. Then so does ${\rm NDL}(\phi)$.
\endlemma
 
 \proof  
 Let $x_0 \in \sX$, $r > 0$, and $B := B(x_0, r)$. Suppose ${\rm UHK}(\phi)$ and ${\rm NLHK}(\phi)$ hold with positive constants $c_1 = c_{\ref{e:2.6}}, c_2 = c_{\ref{e:1.9}}$, $c_3 = c_{\ref{e:phi}. 1}$, $\b_1 = \beta_{\ref{e:phi}. 1}$ and $\e = \e_{\ref{e:1.9}} \in (0, 1)$. Let $\e_0 \in (0, 1 \wedge \e/2 \wedge (1 - \e))$. By the strong Markov property of $X$, for $x, y\in B (x_0, \eps_0 \phi^{-1}(t)) $ and $0 <t \leq \phi (\eps_0 r)$, 
\begin{eqnarray*}
p^B(t, x, y) &=& p(t, x, y) -\bE_y \left[ p(t-\tau_B, x, X_{\tau_B}); \tau_B<t\right] \\
&\geq & \frac{c_2} {V(x, \phi^{-1}(t))} - c_1\, \bE_y \left[  \frac{t }{  V(x, d(x, X_{\tau_B}) \phi ( d(x, X_{\tau_B})))}; \tau_B<t\right]  \\
&\geq & \frac{c_2} {V(x, \phi^{-1}(t))} -    \frac{c_1 \phi (\eps_0 r) }{  V(x, (1-\eps)r) \phi ( (1-\eps)r)}    \\
&\geq & \frac{c_2} {V(x, \phi^{-1}(t))} -    \frac{c_1  c_3^{-1}   }{  V(x,  \phi^{-1}(t))  }  
   \left( \frac{  \eps_0  }{  1-\eps }\right)^{\b_1} ,
 \end{eqnarray*}
 where the last inequality is due to  \eqref{e:phi} and the fact that 
 $$
 (1-\eps) r\geq  (1-\eps)\eps_0^{-1} \phi^{-1}(t) \geq \phi^{-1}(t) .
 $$ 
 Taking $\eps_0$ small so that $c_1c_3^{-1} \left(\eps_0  / ( 1-\eps )\right)^{\b_1} < c_2/2$, we have for every $t \in (0, \phi (\e_0r))$ and $x, y \in B(x_0, \e_0\phi^{-1}(t))$, 
 \begin{equation}\label{e:2.19a}
 p^B(t, x, y) \geq \frac{c_2/2}{V(x, \phi^{-1}(t))}.
 \end{equation}
 Let $c_4 = c_{\ref{e:1.2a}}$ and $\b_2 = \b_{\ref{e:1.2a}. 2}$. Then
 \[
 V(x, \phi^{-1}(t)) \le V(x_0, (1 + \e_0)\phi^{-1}(t)) \le c_4(1 + \e_0)^{\b_2}V(x_0, \phi^{-1}(t).
 \]
 Combining this with \eqref{e:2.19a}, we see that
 \begin{equation}\label{e:2.19}
 p^B(t, x, y) \ge \frac{c_5}{V(x_0, \phi^{-1}(t))}
 \end{equation}
 for any $t \in (0, \phi(\e_0r))$ and $x, y \in B(x_0, \e_0\phi^{-1}(t))$, where $c_5 = c_2(2c_4(1 + \e_0)^{\b_2})^{-1}$. This shows that ${\rm NDL}(\phi)$ holds. 
 \qed 

 \medskip
 
 We next present some sufficient conditions for  ${\rm UHK}(\phi)$ and ${\rm NLHK} (\phi)$.
First, we  need some definitions.

\begin{definition}\label{GT.def10}
\begin{enumerate}
\item[\rm (i)]  We say that ${\rm J}_{\phi, \leq}$ holds if there is some constant $c_{\ref{e:1.2}} > 0$ so that 
\be\label{e:1.2} 
  J(x, y) \le \frac{c_{\ref{e:1.2}}}{V(x,d(x, y)) \phi (d(x, y))}
\ee 
for any $x, y \in \sX$.

 \item[\rm (ii)]\,\,Let $D$ be an open subset of $\sX$. Define
\[
\sF_D := \{ u : u \in \sF, u = 0 \,\,\text{for $\sE$-quasi everywhere on $D^c$}\}
\]
and
\[
\lambda_1(D) = \inf \left\{ \sE(f,f):  \,  f \in \sF_D \hbox{ with }  \|f\|_2 =1 \right\},
\]
where $\| u\|_2$ is the $L^2$-norm of $u \in L^2(K, \mu)$. Note that $\lambda_1(D)$ is the bottom of the Dirichlet spectrum for the generator of the subprocess $X^D$ of $X$ killed upon leaving $D$.

\item[\rm (iii)]\,\,
 We say that an MMD space $(\sX,d,m,\sE, \sF)$  satisfies the {\em Faber-Krahn
inequality} ${\rm FK}(\phi)$, if there exist constants $c_{\ref{e:fki}} > 0$, $\sigma_{\ref{e:fki}} \in (0, 1/2)$ and
$\theta > 0$ such that for any $r \in (0, {\sigma_{\ref{e:fki}}}\,\diam{\sX})$, any ball $B(x, r)$ and any open set $D \subseteq B(x, r)$, 
\be \label{e:fki}
 \lambda_1(D) \ge \frac{c_{\ref{e:fki}}}{\phi(r)} \Big(\frac{V(x,r)}{m(D)}\Big)^{\theta}.
\ee

 \item[\rm (iv)]\,\, Let $\sF_b = \sF\cap L^\infty(\sX, m)$. We say that an MMD space $(\sX,d,m,\sE, \sF)$   satisfies the 
 {\em weak Poincar\'e inequality}
(${\rm PI} (\phi)$)
  if there exist constants $c_{\ref{eq:PIn}} > 0 $, $\sigma_{\ref{eq:PIn}} \in (0, 1/2)$ and $\delta \ge 1$ such that
for any $r \in (0, \sigma_{\ref{eq:PIn}}\,\diam{\sX})$, any $x \in \sX$ and $f \in \sF_b$,
\begin{equation}\label{eq:PIn}
\int_{B_r} (f(x) - \bar{f}_{B_r})^2m(dx) \le c_{\ref{eq:PIn}}\phi(r)\int_{{B_{\delta r}}\times {B_{\delta r}}} (f(y)-f(x))^2\,J(dx,dy),
\end{equation}
 where $B_r = B(x, r)$ and $\bar{f}_{B_r}= \frac{1}{m({B_r})}\int_{B_r} f(x)m(dx)$ is the average value of $f$ on $B_r$.
\end{enumerate} 
\end{definition}

\bigskip

\begin{remark}\label{R:2.9} \rm   
 Under the assumption of VD, \eqref{e:1.2a} and \eqref{e:phi}, by  \cite[Propositions 7.3 and  7.6]{CKW1}, 
 ${\rm PI} (\phi)$ implies  ${\rm FK} (\phi)$.
\end{remark}

For any $f\in \sF_b$,  there is a unique 
Borel measure $\mu_{\<f\>}$ (called the \emph{energy measure} of $f$) on $\sX$ satisfying
$$
\int_\sX g \, d\mu_{\<f\>}=\sE(f, fg)-\frac 12\sE(f^2,g) \quad \hbox{for any }
  g\in \sF_b;
$$
see \cite{CF, FOT}.

\begin{definition}(\cite[Definition 1.5]{CKW1})
  \rm Let $U$ and $V$ be open subsets sets of $\sX$ with $\overline{U} \subseteq V$. A measurable function $\vp: \sX \to [0, 1]$ is called a cut-off function for $U \subseteq V$ if $\vp \equiv 1$ on $U$ and $\vp \equiv 0$ on $V^c$. We say that the condition $\CSJ(\phi)$ holds if there exist constants $C_0\in (0,1]$ and $C_1, C_2>0$ such that for every $0 < r \le R$, almost all $x_0 \in \sX$ and any $f \in \sF$, there exists
a cut-off function $\vp \in \sF_b$ for $B(x_0,R) \subset B(x_0,R+r)$ so that the following holds:
\be \label{e:csj1} \begin{split}
 \int_{B(x_0,R+(1+C_0)r)} f^2 \, d\mu_{\<\vp\>}
\le \,\,&C_1 \int_{U\times U^*}(f(x)-f(y))^2\,J(dx,dy) \\
&+ \frac{C_2}{\phi(r)}  \int_{B(x_0,R+(1+C_0)r)} f(x)^2m(dx),
\end{split}
\ee
where $U=B(x_0,R+r)\setminus B(x_0,R)$ and $U^*=B(x_0,R+(1+C_0)r)\setminus B(x_0,R-C_0r)$.
\end{definition}

 \begin{remark}\label{R:2.11}  \rm 
 If $\beta_{{\ref{e:phi}. 2}} < 2$, where $\beta_{{\ref{e:phi}. 2}}$ appears in \eqref{e:phi}, then ${\rm CJS}(\phi)$ condition  is automatically satisfied under condition $J_{\phi, \leq}$ in view of \cite[Remark 1.7]{CKW1}. 
\end{remark}

 The following theorem is established in \cite[Theorem 1.15]{CKW1}.   
 
\begin{thm} {\rm (\cite[Theorem 1.15]{CKW1})} \label{T:2.12}
  Suppose that the metric space $(\sX, d)$ is unbounded.  
  The following are equivalent:
\begin{enumerate}
\item[\rm (i)]  ${\rm FK}(\phi)$, ${\rm J}_{\phi,\le}$ and ${\rm CSJ} (\phi)$ hold.

 \item[\rm (ii)]    $(\sE, \sF)$ is conservative   and  ${\rm UHK}(\phi)$ holds 
 \end{enumerate} 
If one of them holds, then there is a constant $c\geq 1$ so that for all $x\in \sX$ and $r>0$, 
\begin{equation} \label{e:E}
c^{-1} \phi (r) \leq \bE_x [ \tau_{B(x, r)}] \leq c \phi (r),
\end{equation} 
where $\tau_{B(x, r)}:=\inf\{t>0: X_t\notin B(x, r)\}$ is the first exit time from the ball $B(x, r)$ by the process $X$. 
\end{thm}

\medskip

\begin{thm}\label{T:2.13}
 Suppose that the metric space $(\sX, d)$ is unbounded.  
 Suppose that ${\rm PI}(\phi) $, ${\rm J}_{\phi,\le}$ and ${\rm CSJ} (\phi)$ hold. 
 Then 
  \begin{enumerate}
  \item[\rm (i)]    $(\sE, \sF)$ is conservative   and ${\rm UHK}(\phi)$ holds;

\item[\rm (ii)]   ${\rm NDL}(\phi)$   holds.
 \end{enumerate} 
\end{thm}

\proof (i) This follows from Theorem \ref{T:2.12} and the fact that ${\rm PI} (\phi)$ implies  ${\rm FK} (\phi)$.

\medskip

(ii)  This follows from  \eqref{e:E} and \cite[Propositions 4.14 and 4.10]{CKW2}. 
  \qed

\bigskip

\remark \label{R:2.10}  \rm 
\begin{enumerate}
\item[(i)]
It is shown in \cite[Proposition 3.5]{CKW2} that ${\rm NDL}(\phi)$ implies ${\rm PI}(\phi)$.
Note that although it is assumed $\sX$ is unbounded in \cite{CKW2}, this result holds for bounded $\sX$ as well by 
the same proof. 
 This combined with Lemma \ref{L:2.9} yields that ${\rm UHK}(\phi)$ and     ${\rm NLHK}(\phi)$  implies  ${\rm PI}(\phi)$. 

\item[(ii)]  In fact, all the main results in \cite{CKW1, CKW2} hold when $\sX$ is bounded as well;  see \cite[Remark 8.3]{CC} for details.
 Thus the unboundedness 
condition of $(\sX, d)$ in Theorems \ref{T:2.12} and \ref{T:2.13} as well as in Theorem \ref{T:2.11} below  can be dropped. 
\end{enumerate}
\endremark

Combining Remark~\ref{R:2.10} with Remark~\ref{R:2.9}, Theorem  \ref{T:2.13} and Lemma \ref{L:2.9} immediately yields the following.

\begin{thm}\label{T:2.11}
  Suppose that   ${\rm J}_{\phi,\le}$ and ${\rm CSJ} (\phi)$ hold. 
 Then the following are equivalent. 
  \begin{enumerate}
  \item[\rm (i)]   ${\rm PI}(\phi) $ holds;

\item[\rm (ii)]   $(\sE, \sF)$ is conservative,   and both ${\rm UHK}(\phi)$ and  ${\rm NLHK}(\phi)$   hold;

\item[\rm (iii)]   $(\sE, \sF)$ is conservative,  and both   ${\rm UHK}(\phi)$ and  ${\rm NDL}(\phi)$ hold.
 \end{enumerate} 
\end{thm}

\setcounter{equation}{0} 
\section{Main Theorem: lower bound of heat kernel}

In this section, we present our main theorem, Theorem~\ref{GT.thm10}, of this paper. It concerns with a lower bound estimate of a hear kernel associated with an MMD space via averaged jump kernel.

\thm\label{GT.thm10}
Let $(\sX, d, m, \sE, \sF)$ be a MMD. Suppose that $(\sE, \sF)$ is conservative, and ${\rm UHK}(\phi)$ and ${\rm NLHK} (\phi)$ hold. Then there exist constants $c_{\ref{GT.eq10}} > 0$, $\e_{\ref{GT.eq10}} \in (0, \e_{\ref{e:1.9}}/2)$ and $\kappa \in (0, 1)$ such that 
\begin{equation}\label{GT.eq10}
p(t, x, y) \ge c_{\ref{GT.eq10}}tJ(x, y, \e_{\ref{GT.eq10}}\phi^{-1}(t))
\end{equation}
for any $(t, x, y) \in (0, \kappa\phi (\diam (\sX))) \times \sX \times \sX$ with $d(x, y) > \e_{\ref{e:1.9}}\phi^{-1}(t)$, where
\[
J(x, y, r) 
 = \frac {1}{V(x, r)V(y, r)}\int_{B(x, r)}\int_{B(y, r)} J(u, v)m(du)m(dv)
\]
Here the constant $\kappa \in (0, 1)$ is a constant that depends only on the parameters in (VD), \eqref{e:phi} for $\phi$ and ${\rm UHK}(\phi)$.
\endthm

 \demo  First we are going to show several claims.

\smallskip

\noindent{\rm \bf Claim 1}\,\,There is a constant $c_1 \geq 1$ such that for any $t, r > 0$ all $x\in \sX$,
\begin{equation}\label{e:1.12}
\int_{B(x,r)^c} p(t, x,y)  \,m(dy)\le \frac{c_1 t}{\phi(r)}.
\end{equation}

\smallskip

\noindent Proof of Claim 1:
It suffices to show \eqref{e:1.12} for $\phi (r)>t$, as it holds trivially when $\phi (r) \leq t$.
By ${\rm UHK}(\phi)$, \eqref{e:1.2a}, \eqref{e:phi} and  \eqref{e:2.6}, 
\begin{eqnarray*}
\int_{B(x,r)^c} p(t, x,y)  \,m(dy) 
&= & \sum_{k=0}^\infty \int_{B(x, 2^{k+1}r) \setminus B(x, 2^k r)}  p(t, x,y)  \,m(dy) \\
&\leq &  \sum_{k=0}^\infty  \frac{ c_2 t}{V(x, 2^kr) \phi (2^kr)} V(x, 2^{k+1}r)  \\
&=& \frac{ c_2 t}{\phi (r)}   \sum_{k=0}^\infty    \frac{\phi (r)}{\phi (2^kr)} \frac{V(x, 2^{k+1}r)}{V(x, 2^kr)}  \\ 
&\leq &  \frac{ c_3 t}{\phi (r)} \sum_{k=0}^\infty    2^{- \beta_3 k} .
\end{eqnarray*} 
This proves the claim \eqref{e:1.12}.

\medskip

\noindent{\rm \bf Claim 2}\,\,
There exists  a constant $\kappa \in (0, 1)$   such that
\begin{equation}\label{e:1.16}
\bP_x  \left(\tau_{B(x,  \phi^{-1} (  t)/\kappa )}>   t \right)  \geq 1/2 
\end{equation} 
for any $x \in X$ and  $0 < t < \kappa\phi(\diam (\sX))$.

\smallskip

\noindent Proof of Claim 2: Since the Hunt process $X$ is conservative,
 by the strong Markov property of $X$ and \eqref{e:1.12}, for any $t > 0$ and $0<r < \diam (\sX)$ and $x\in \sX$,
\begin{eqnarray*}
&& \bP_x(\tau_{B(x,r)}\le t ) \\
&=&\bP_x(\tau_{B(x,r)}\le t, X_{2t}\in B(x,r/2)^c)+\bP_x(\tau_{B(x,r)}\le t, X_{2t}\in B(x,r/2))\\
&\le&\bP_x( X_{2t}\in B(x,r/2)^c)+ \sup_{z\notin B(x,r)^c, s\le t} \bP_z( X_{2t-s}\in B(z,r/2)^c)\\
&\le & \frac{4 c_3 t}{\phi(r/2)} \leq \frac{c_4 t}{\phi (r)}. 
 \end{eqnarray*}  
In view of \eqref{e:phi}, there is a constant $\kappa \in (0, 1)$ so that 
\begin{equation}
\phi (\phi^{-1} (t)/\kappa) \geq 2c_4 t  \quad \hbox{for every }  t>0.
\end{equation} 
 Thus  we have for any $0< t< \kappa\phi (\diam (\sX))$,
 \begin{equation} \label{e:1.14}
 \bP_x  \left(\tau_{B(x, \phi^{-1} (t)/\kappa)}>  t \right)  
 =1- \bP_x  \left(\tau_{B(x, \phi^{-1} (t)/\kappa)}\leq  t \right)
  \geq 1/2.
 \end{equation} 
This proves Claim 2. 

 \medskip

By  Lemma \ref{L:2.9},  ${\rm NDL}(\phi)$ holds; that is, there are constants $c_5 > 0$ and $\eps_0 \in (0, 1)$ so that for every $x_0\in \sX$, $r>0$, $t\in (0, \phi (\eps_0r))$ and $B = B(x_0, r)$, 
 \begin{equation}\label{e:2.15a}
 p^B(t, x, y) \geq   \frac{c_5}{V(x_0, \phi^{-1}(t))}  \quad \hbox{for } x, y \in B(x_0, \eps_0 \phi^{-1}(t)).
 \end{equation}
 By decreasing the value of $\eps_0$ if needed, by \eqref{e:phi} we may and do assume 
 $\eps_0\in (0, \e_{\ref{e:1.9}}/2)$ so that 
 \begin{equation}\label{e:2.16a}
 2\eps_0 \phi (t)\leq \e_{\ref{e:1.9}}\phi^{-1}(t/2) \quad \hbox{for every} t > 0.
 \end{equation}

\medskip

\noindent{\rm \bf Claim 3}\,\,There exist constants  $\e_1 \in (0, \e_0)$, $\lambda_0 \in (0, 1/2)$ and $c_6 > 0$ such that 
\[
p^{B(x, \e_0\phi^{-1}(t))}(s, x, z) \ge \frac{c_6}{V(x, \phi^{-1}(t))}
\]
for any $x \in \sX$, $t > 0 $,
$z \in B(x, \e_1\phi^{-1}(t))$ and $s \in [\lambda_0 t/2, \lambda_0 t]$.

\smallskip

\noindent Proof of Claim 3:\,\,
By   \eqref{e:phi}, 
  there is a constant $\lambda_0\in (0, 1/2)$ so that 
  \begin{equation}\label{e:2:18} 
 \phi^{-1}( \lambda_0 t  )\leq  
  \min\{\kappa\eps_0, \eps_0^2\}  \, \phi^{-1}( t) 
   \quad \hbox{for any } t>0.
  \end{equation}
By  \eqref{e:phi}  again, there is some $\e_* \in (0, \e_0)$ so that 
\begin{equation}\label{e:2.17}
 \e_*\phi^{-1}(t) \leq \e_0\phi^{-1}(\lambda_0t/2)  \quad \hbox{for any } t>0.
 \end{equation}
Thus by \eqref{e:2.15a}-\eqref{e:2.17} and \eqref{e:1.2a}-\eqref{e:phi}, if $\e_1 \in (0, \e_*)$, then we have 
\[
p^{B(x, \e_0\phi^{-1}(t))}(s, x, z) \ge \frac{c_5}{V(x, \phi^{-1}(s))} \ge \frac{c_5}{V(x, \phi^{-1}(\lambda_0t))} \ge \frac{c_6}{V(x, \phi^{-1}(t))}
\]
for any $s \in [\frac {\lambda_0{t}}2, \lambda_0t]$ and $z \in B(x, \e_1\phi^{-1}(t))$, proving the Claim 3.

\medskip
Let $r_j(t) = \e_j\phi^{-1}(t)$ and $B_j(z) = B(z, r_j(t))$ for $j = 0, 1$ and $z \in \sX$. Define
\[
J_2(u, v, r) = \frac 1{V(v, r)}\int_{B(v, r)}J(u, z)m(dz)
\]
for $u, v \in \sX$ and $r > 0$. Using the above established claims with \eqref{e:1.2a} and \eqref{e:phi}, for any $0 < t < \kappa\phi(\diam (\sX))$, we have that
 \begin{align*}
&\bP_x \left( X_{\lambda_0 t}\in B(y,  2 \eps_0 \phi^{-1}(t) ) \right) \\
&\geq \bP_x\big(\hbox{$X$ hits $B_0(y)$ before $\lambda_0 t$ and then stays }\\
& \hskip 0.6 truein  \hbox{within its $\eps_0 \phi^{-1}(t)$ neighborhood for at least $\lambda_0  t$ units of time} \big) \\
&\geq \bP_x(\sigma_{B_0(y)} < \lambda_0t)\inf_{z\in B(y, \eps_0  \phi^{-1}(t))}\bP_z(\tau_{B(z, \eps_0 \phi^{-1}(t))}\geq \lambda_0  t)  \\
&\geq  \frac12 \bP_x  \left(X_{(\lambda_0t) \wedge \tau_{B_0(x)}} \in  B_0(y)\right)\\
&= \frac12 \bE_x\Big[\int_0^{(\lambda_0  t) \wedge\tau_{B_0(x)}}\int_{B_0(y)}J(X_s, v)m(dv)ds\Big] \\
&\ge \frac12 \bE_x\Big[\int_0^{(\lambda_0  t) \wedge\tau_{B_0(x)}}\int_{B_1(y)}J(X_s, v)m(dv)ds\Big] \\
&= \frac 12m(B_1(y))\bE_x\Big[\int_0^{(\lambda_0  t) \wedge\tau_{B_0(x)}}J_2(X_s, y, r_1 (t))ds\Big]\\
&= \frac 12m(B_1(y))\int_0^{\lambda_0t}\int_{B_0(x)}p^{B_0(x)}(s, x, z)J_2(z, y, r_1 (t))m(dz)ds\\
&\ge \frac 12m(B_1(y))\int_{\frac{\lambda_0t}2}^{\lambda_0t}\int_{B_1(x)}p^{B_0(x)}(s, x, z)J_2(z, y, r_1 (t))m(dz)ds\\
&\ge c_7V(y, \phi^{-1}(t))tJ(x, y, r_1 (t)),
\end{align*}
where the first equality is due to \eqref{GT.eq15}. Note that by \eqref{e:2.16a},
\[
2\e_0\phi^{-1}(t) < \e_{\ref{e:1.9}}\phi^{-1}(t/2) \le \e_{\ref{e:1.9}}\phi^{-1}((1 - \lambda_0)t).
\]
For any $t \in (0, \phi (\diam (\sX))$ and $x, y \in U$ with $d(x, y) > \e_{\ref{e:1.9}} \phi^{-1}(t)$, the above together with ${\rm NLHK} (\phi)$ yields
 \begin{eqnarray*}
p(t, x, y) &\geq& \int_{B(y, 2\eps_0 \phi^{-1}(t))} p(\lambda_0  t,x,z) p((1-\lambda_0 ) t, z, y)m(dz)\\
&\geq&  \frac{c_{\ref{e:1.9}}}{V(y, \phi^{-1}((1 - \lambda_0)t))}\bP_x(X_{\lambda_0 t}\in B(y, 2 \eps_0 \phi^{-1}(t)))\\
&\geq& c_7c_{\ref{e:1.9}}\, \frac{ V(y, \phi^{-1}(t))}{V(y, \phi^{-1}((1 - \lambda_0)t))}tJ(x, y, r_1 (t)) \\
& \geq&  c_8tJ(x, y, r_1 (t)).
  \end{eqnarray*}
This completes the proof of the theorem.
 \enddemo

 The following corollary follows immediately from Theorem~\ref{T:2.11} and Theorems~\ref{GT.thm10}.
 
 \begin{cor} \label{C:2.16}
 Suppose that   ${\rm J}_{\phi,\le}$, ${\rm CSJ} (\phi)$ and ${\rm PI}(\phi) $ hold for the MMD  \hfill \break
  $(\sX, d, m, \sE, \sF)$.
 Then $(\sE, \sF)$ is conservative and has a  jointly continuous density function p(t, x, y) on
$(0, \infty) \times \sX \times \sX$ with respect to $m$ that satisfies   ${\rm UHK}(\phi)$,   ${\rm NDL}(\phi)$  as well as the lower bound off-diagonal estimate \eqref{GT.eq10}. 
 \end{cor}

Below is an example, taken from \cite[Example 1.2]{CKW2}, of 
a MMD $(\sX, d, m, \sE, \sF)$ with degenerate jump kernel $J(x, y)$ for which the conditions of Corollary \ref{C:2.16} are satisfied. 
 
\begin{example}\label{E:3.3}
  Let $(\sX, d)$ be   $\R^d$ equipped with Euclidean metric, and $m$ be the Lebesgue measure on $\R^d$.
For $0 < \theta  < \pi/2$ and $v \in \R^d$ with $|v|=1$, let $C$ be the two-sided circular cone with axis in the direction of $v$,
 apex angle $2\theta$ and vertex at
the origin, that is $C=\{z\in \R^d:  |z\cdot v|/|z|> \cos \theta \}$. For $0<\alpha <2$, define $\phi (r)=r^\alpha$ and 
$$
J(x,y)=\1_C (x-y)   |x-y|^{-d-\alpha} \quad \hbox{for } x, y\in \R^d.
$$
Denote by $(\sE, \sF)$ the corresponding regular Dirichlet form defined by \eqref{GT.eq100}. 
Clearly $J_{\phi, \leq}$ holds. As illustrated in \cite[Example 1.2 (continued) on p.3797]{CKW2},  both
${\rm CSJ} (\phi)$ and ${\rm PI}(\phi) $ hold for   $(\sX, d, m, \sE, \sF)$.
Hence the conclusions of Corollary \ref{C:2.16} hold for this MMD $(\sX, d, m, \sE, \sF)$.
\end{example}

A similar example but with the cone axis dependent on the location can be found in \cite[Example 1.3]{CKW2},
for which the conditions of  Corollary \ref{C:2.16} are also satisfied.

\medskip

Under the UJS condition given below, Theorem~\ref{GT.thm10} implies the lower transition density estimate \eqref{e:1.5} as we will see in Corollary~\ref{GT.cor100}. It is however no better than \eqref{GT.eq10} in general.
\definition\label{UJS}
A jump kernel $J$ is said to satisfy {\rm UJS} if there exists $c_{\ref{GT.eq20}} > 0$ such that
\begin{equation}\label{GT.eq20}
J(x, y) \le \frac{c_{\ref{GT.eq20}}}{V(x, r)}\int_{B(x, r)} J(u, y)m(du)
\end{equation}
for any $x, y \in \sX$ and $r \in (0,  d(x, y)/2)$.
\enddefinition

\remark \rm 
The UJS condition was first introduced  \cite{BBK} in the setting of graphs, and then in \cite{CKK}
 for the general setting of metric measure spaces. It is established in \cite[Theorem 1.18]{CKW2} that, under VD and RVD,   UJS combined with ${\rm NDL}(\phi)$ is equivalent to the parabolic Harnack inequality ${\rm PHI}(\phi)$.
\endremark

\lemma\label{GT.lemma10}
Assume that {\rm UJS} and {\rm VD} are satisfied. Then there exists $c_{\ref{GT.eq30}} > 0$ such that
\begin{equation}\label{GT.eq30}
J(x, y) \le c_{\ref{GT.eq30}}J(x, y, r)
\end{equation}
for any $x, y \in \sX$ and $r \in (0,  d(x, y)/2)$.
\endlemma

\demo
Suppose that  $0<r < d(x, y)/2$. Then for any $v \in B(y, r)$,
\[
d(x, v) \ge d(x, y) - r > 2r - r = r.
\] 
Thus by UJS, 
\begin{eqnarray}\label{GT.eq300}
J(x, v) & \le & \frac{c_{\ref{GT.eq20}}}{V(x, r/2)}\int_{B(x, r/2)} J(u, v)m(du) \nonumber \\
&\le & \frac{c}{V(x, r)}\int_{B(x, r)} J(u, v)m(du),
\end{eqnarray}
where the constant $c$ depends on the constants in (VD) and $c_{\ref{GT.eq20}}$. By UJS again, it follows from \eqref{GT.eq300} that
 \begin{multline*}
J(x, y) \le \frac{c_{\ref{GT.eq20}}}{V(y, r)}\int_{B(y, r)} J(x, v)m(dv) \\
\le \frac{c_{\ref{GT.eq20}}c}{V(y, r)V(x, r)}\int_{B(y, r)}\int_{B(x, r)} J(u, v)m(du)m(dv) = c_{\ref{GT.eq20}}J(x, y, r).
\end{multline*}
\enddemo

\cor\label{GT.cor100}
Suppose that ${\rm UHK}(\phi)$, ${\rm NLHK} (\phi)$ and {\rm UJS} hold for the MMD $(\sX, d, m, \sE, \sF)$, and $(\sE, \sF)$ is conservative. Then there exist positive constants $c_{\ref{GT.eq50}}$ and $\kappa \in (0, 1)$ such that
\begin{equation}\label{GT.eq50}
p(t, x, y) \ge c_{\ref{GT.eq50}}tJ(x, y)
\end{equation}
for any $(t, x, y) \in (0, \kappa\phi (\diam (\sX))) \times \sX \times \sX$ with $d(x, y) > \e_{\ref{e:1.9}}\phi^{-1}(t)$,
\endcor

\demo
Let $r = \e_{\ref{GT.eq10}}\phi^{-1}(t)$. Since $\e_{\ref{GT.eq10}} < \e_{\ref{e:1.9}}/2$, we see that
\[
d(x, y) > \e_{\ref{e:1.9}}\phi^{-1}(t) > 2r.
\]
The desired inequality follows directly from \eqref{GT.eq10} and Lemma~\ref{GT.lemma10}.
\enddemo

As is mentioned above, the estimate \eqref{GT.eq50} is simpler but no shaper than \eqref{GT.eq10} because of \eqref{GT.eq30}. In fact, if $m(\partial{B(x, r)}) = 0$ for any $x \in \sX$ and $r > 0$, then $J(x, y, r)$ is continuous for any $(x, y, r) \in \sX \times \sX \times (0, \infty)$. Hence in the case where $J(x, y)$ is not continuous, \eqref{GT.eq10} is sharper than \eqref{GT.eq50}. Despite this, the UJS condition still has an advantage, that is, we can make the parameter $\kappa$ in Theorem~\ref{GT.thm10} as large as possible.

\cor\label{GT.cor200}
Suppose that ${\rm UHK}(\phi)$, ${\rm NLHK} (\phi)$ and {\rm UJS} hold for the MMD $(\sX, d, m, \sE, \sF)$, and $(\sE, \sF)$ is conservative. Then for any $\kappa > 0$, there exist constants $c_{\ref{GT.eq70}} > 0$ and  $\e_{\ref{GT.eq70}} \in (0, \e_{\ref{e:1.9}}/2)$ such that 
\begin{equation}\label{GT.eq70}
p(t, x, y) \ge c_{\ref{GT.eq70}}tJ(x, y, \e_{\ref{GT.eq70}}\phi^{-1}(t))
\end{equation}
for any $(t, x, y) \in (0, \kappa\phi (\diam (\sX))) \times \sX \times \sX$ with $d(x, y) > \e_{\ref{e:1.9}}\phi^{-1}(t)$
\endcor

This corollary is essentially due to the fact that the combination of ${\rm UHK}(\phi)$, ${\rm NLHK} (\phi)$ and {\rm UJS} implies the parabolic Harnack inequality. See its proof below for the details.\par
We use the following lemma in the proof of Corollary~\ref{GT.cor200}.

\lemma\label{GT.lemma20}
Let $\b = \b_{\ref{e:1.2a}, 2}$ and $c = c_{\ref{e:1.2a}}$ appearing in \eqref{e:1.2a}. Then
\[
J(x, y, r) \le c^2\Big(\frac Rr\Big)^{2\b}J(x, y, R)
\]
for any $x, y \in \sX$ and $r, R > 0$ with $0 < r < R$.
\endlemma

\demo
Using \eqref{e:1.2a}, we see that
\[
J(x, y, r) \le \frac{V(x, R)V(y, R)}{V(x, r)V(y, r)}J(x, y, R) \le c^2\Big(\frac Rr\Big)^{2\b}J(x, y, R).
\]
\enddemo

\demo[Proof of Corollary~\ref{GT.cor200}]
By Theorem~\ref{T:2.11}, NDL$(\phi)$ holds for $X$ as well under ${\rm UHK}(\phi)$ and ${\rm NLHK} (\phi)$. Thus by \cite[Theorem 1.18]{CKW2}, we have the parabolic Harnack inequality ${\rm PHI}^+(\phi)$. In particular, there are positive constants $c_{\ref{PHI}}$ and $C_{\ref{PHI}}$ such that 
\begin{equation}\label{PHI}
\sup_{ (s, x) \in Q_-}  p(s, x, y)
 \leq C_{\ref{PHI}} \,  \inf_{ (t, x)\in Q_+} p(t, x, y)
\end{equation}
for any $x_0, y \in \sX$ and $r > 0$, where
\begin{align*}
Q_- &= [c_{\ref{PHI}}\phi(r), 2c_{\ref{PHI}}\phi(r)] \times B(x_0, r)\quad\text{and}\\
Q_+ &= [3c_{\ref{PHI}}\phi(r), 4c_{\ref{PHI}} \phi(r)] \times B(x_0, r).
\end{align*}
For any $t$, choose $r > 0$ such that $(x, t) \in Q_+$. Then $(x, t/2) \in Q_-$ and \eqref{PHI} shows
\begin{equation}\label{GT.eq500}
C_{\ref{PHI}}p(t, x, y) \geq p(t/2, x, y).
\end{equation}
Now denote $\kappa \in (0, 1)$ in Theorem~\ref{GT.thm10} by $\kappa_0$. If $\kappa$ in Corollary~\ref{GT.cor200} satisfies $\kappa \le \kappa_0$, then the corollary follows immediately from Theorem~\ref{GT.thm10}. Assume that $\kappa > \kappa_0$. Let $\delta = \kappa_0/\kappa$. Using \eqref{GT.eq500} repeatedly if necessary, we see that there exists $c_1 > 0$ such that 
\[
p(t, x, y) \ge c_1p({\delta}t, x, y)
\]
for any $x, y \in \sX$ and $t > 0$ Suppose $0 < t \le \kappa\phi(\diam(\sX))$. Then ${\delta}t \le \kappa_0\phi(\diam{\sX})$. Theorem~\ref{GT.thm10} yields
\begin{equation}\label{GT.eq80}
p(t, x, y) \ge c_1p({\delta}t, x, y) \ge c_1c_{\ref{GT.eq10}}tJ(x, y, \e_{\ref{GT.eq10}}\phi^{-1}({\delta}t))
\end{equation}
Choose $\e_{\ref{GT.eq70}}$ such that $\e_{\ref{GT.eq10}}\phi^{-1}({\delta}t) \ge \e_{\ref{GT.eq70}}\phi^{-1}(t)$ for any $t > 0$. Then by Lemma~\ref{GT.lemma20}, 
\begin{eqnarray}\label{GT.eq90}
J(x, y, \e_{\ref{GT.eq10}}\phi^{-1}({\delta}t)) 
&\ge&  c^{-2}\Big(\frac{\e_{\ref{GT.eq70}}\phi^{-1}(t)}{\e_{\ref{GT.eq10}}\phi^{-1}({\delta}t)}\Big)^{2\b}J(x, y, \e_{\ref{GT.eq70}}\phi^{-1}(t))
\nonumber \\ 
&\ge & c^{-2}\Big(\frac{\e_{\ref{GT.eq70}}}{\e_{\ref{GT.eq10}}}\Big)^{2\b}J(x, y, \e_{\ref{GT.eq70}}\phi^{-1}(t))
\end{eqnarray}
Combining \eqref{GT.eq80} and \eqref{GT.eq90}, we obtain the desired estimate.
\enddemo

\medskip

In the rest of this section, we are going to consider the consistency between near-diagonal and off-diagonal lower bounds in Theorem~\ref{GT.thm10}. Namely, when $d(x, y) = \e_{\ref{e:1.9}}\phi^{-1}(t)$, the near-diagonal lower bound $1/V(x, \phi^{-1}(t))$ and the off-diagonal lower bound $tJ(x, y, \e_{\ref{GT.eq10}}\phi^{-1}(t))$ are consistent
or not. More precisely, whether their ration are uniformly bounded from below and above by some positive constants throughout the space $\sX \times \sX$ or not. Substituting $\phi^{-1}(t) = d(x, y)/\e_{\ref{e:1.9}}$ and taking \eqref{e:1.2a} and \eqref{e:phi} into account, we end up with comparison between 
\[
\frac 1{V(x, d(x, y))}\quad\text{and}\quad \phi(d(x, y))J(x, y, {\rho}d(x, y))
\]
for $x, y \in \sX$, where $\rho = \e_{\ref{GT.eq10}}/\e_{\ref{e:1.9}} \in (0, 1/2)$. Recall that $O < \e_{\ref{GT.eq10}} < \e_{\ref{e:1.9}}/2$ by Theorem~\ref{GT.thm10}.\par
Under the condition of Theorem~\ref{GT.thm10}, it is easy to see that one direction of comparison always holds. Indeed, by comparing \eqref{e:2.6} and \eqref{GT.eq10}, we have the following fact.

\prop\label{GT.prop10}
Suppose that $(\E, \F)$ is conservative and ${\rm UHK}(\phi)$ holds, then there exists $c_{\ref{GT.eq110}} > 0$ such that
\begin{equation}\label{GT.eq110}
\frac 1{V(x, d(x, y))} \ge c_{\ref{GT.eq110}}\phi(d(x, y))J(x, y, {\rho}d(x, y)).
\end{equation}
\endprop

However, the other direction may fail as we will see in Proposition~\ref{TSS.prop20}. So, the off-diagonal lower bound estimate \eqref{GT.eq10} is not always best possible so that there is room for improvement.

\setcounter{equation}{0}
\section{Example: the trace of the standard resistance form on the Sierpinski gasket}\label{S:3}\label{TSS} 

Let $p_0 = (0, 0)$, $p_1 = (1, 0)$ and $p_2 = \big(\frac 12, \frac{\sqrt{3}}2\big)$, and let $S = \{0, 1, 2\}$. For $i \in S$, define $F_i: \BbR^2 \to \BbR^2$ by
\[
F_i(z) = \frac 12(z - p_i) + p_i
\]
for $x \in \BbR^2$. Then it is well-known that there exists a unique non-empty compact set $K \subseteq \BbR^2$ such that
\[
K = F_1(K) \cup F_2(K) \cup F_3(K).
\]
See \cite{AOF} for example. The set $K$ is called the Sierpinski gasket. Let $d_*$ be the restriction of the Euclidean metric on $K$. Note that $(K, d_*)$ is connected and hence uniformly perfect. \par
Define a sequence of graphs $\{(V_m, E_m)\}_{m \ge 0}$ inductively as
\begin{align*}
V_0 &= \{p_0, p_1, p_2\},\\ 
E_0 &= \{(p_i, p_j) : \   i, j \in \{0, 1, 2\}\},\\
V_{m + 1} &= F_0(V_m) \cup F_1(V_m) \cup F_2(V_m),\\
E_{m + 1} &= \{(F_k(p), F_k(q)): \   (p, q) \in E_m, k \in \{1, 2, 3\}\}.
\end{align*}
Moreover, define $\E_m: \ell(V_m) \to [0, \infty)$ by
\[
\E_m(u) = \frac 12\Big(\frac 53\Big)^m\sum_{(p, q) \in E_m} |u(p) - u(q)|^2.
\]
Finally define $(\E, \F)$ by
\[
\F = \{  u \in C(K) :  \text{ $\E_m(u|_{V_m})$ is convergent as $m \to \infty$} \} . 
\] 
and 
\[
\E(u) = \lim_{m \to \infty} \E_m(u)
\]
for $u \in \F$. Then $(\E, \F)$ is known to be a resistance form on $K$ and called the standard resistance form on the Sierpinski gasket. See \cite[Section~2]{AOF} and/or \cite{Ki16} for the definition and the general theory of resistance forms. One can also find a concise introduction of resistance forms in the appendix of \cite{KiKTaka}. Moreover see \cite[Example~3.1.5]{AOF} for the standard resistance form on the Sierpinski gasket. Note that the resistance between two disjoint compact subsets $A$ and $B$ of $K$, $R(A, B)$, associated with $(\E, \F)$ is defined as
\[
R(A, B) = \{\E(u): \  u \in \F, u|_A \equiv 1, u|_B \equiv 0\}.
\]
In this paper, we do not distinguish a set consisting of a single point $x$, i.e $\{x\}$, and a point $x$. For example, we write $R(x, y)$ in place of $R(\{x\}, \{y\})$ for $x, y \in K$. Then $R: K \times K$ is called a resistance metric associated with the resistance form $(\E, \F)$ and it is indeed a metric on $K$. See the references on the resistance forms above. Moreover, it is known that
 
\begin{equation}\label{TSS.eq10}
|x - y|^{\tau} \asymp R(x, y) 
\end{equation}
for any $x, y \in K$, where $\tau$ is given by $\displaystyle \frac 35 = \Big(\frac12\Big)^{\tau}$, i.e.
\begin{equation}\label{e:tau}
\tau = \frac{\log 5 - \log 3}{\log 2}.
\end{equation}
See \cite[(1.6.10)]{StrBook}. 
This in particular implies that  $(K, R)$ is connected, and so $(K, R)$ is uniformly perfect. Furthermore, the resistance form $(\E, \F)$ is local. Then by \cite[Proposition~7.6]{Ki16}, $(\E, \F)$ satisfies the annulus comparable condition (ACC), that is, there exists $\e > 0$ such that
 \[
 R(x, B_R(x, r)^c) \asymp R(x, \overline{B_R(x, (1 + \e)r)} \cap B_R(x, r)^c) 
 \]
 for any $x \in K$ and $r > 0$ with $B_R(x, r) \neq K$, where $B_R(x, r) = \{y| y \in K, R(x, y) < r\}$.
Moreover, by \cite[Theorem~7.12]{Ki16}, we have (RES), which is
\begin{equation}\label{TSS.eq20}
R(x, B_R(x, r)^c) \asymp r \tag{RES}
\end{equation}
for any $x \in K$ and $r \in (0, 1]$. \par
Let $\mu$ be the normalized Hausdorff measure of $(K, d_*)$. Then $(\E, \F)$ is known to be a regular Dirichlet form on $L^2(K, \mu)$. See \cite[Theorem~9.4]{Ki16} for details. The diffusion process associated with the Dirichlet form $(\E, \F)$ coincides with the Brownian motion on the Sierpinski gasket constructed by Kusuoka~\cite{Kus1}, Goldstein~\cite{Go} and Barlow-Perkins~\cite{BP}.  For the case of the infinite Sierpinski gasket, a proof of the corresponding fact is given in \cite{Ki25}. Modifying its arguments, one can identify  the Brownian motion with the diffusion process derived from the Dirichlet form $(\E, \F)$ for the current case, i.e. the (compact) Sierpinski gasket.
 
 \medskip

Set $I = [0, 1] \times \{0\} \subseteq \BbR^2$, which is the line segment $\overline{p_0p_1}$. For simplicity, we identity $I$ with the unit interval in the natural way. We consider the trace   $(\E|_I, \F|_I)$ of $(\E, \F)$ on $I$, whose definitions are
\[
\F|_I = \{v :  \text{ there exists $u \in \F$ such that $u|_I = v$}\}
\]
and
\[
\E|_I(v, v) = \min\{\E(u, u): \   u \in \F \hbox{ with }  u|_I = v\}
\]
for $v \in \F|_I$. By \cite[Theorem~8.4]{Ki16}, $(\E|_I, \F|_I)$ is a resistance form and the associated resistance metric is the restriction of the resistance metric $R$ associated with $(\E, \F)$ to $I$. Note that this resistance form $(\E|_I, \F|_I)$ is no longer local. Using \cite[Theorem~8.6 and Corollary 8.7]{Ki16}, we see that $(\E|_I, \F|_I)$ satisfies (ACC) and \eqref{TSS.eq20}. Let $\nu$ be the restriction of the $1$-dimensional Lesbegue measure onto $I$. Then $(\E|_I, \F|_I)$ is a regular Dirichlet form on $L^2(I, \nu)$. Let $(\{X_t\}_{t > 0}, \{P_x\}_{x \in I})$ be the associated Hunt process, which is the trace process of the Brownian motion on the Sierpinski gasket $K$ associated with the strongly local regular Dirichet form $(\sE, \sF)$ on $L^2(K, \mu)$. By \cite[Theorem~10.4]{Ki16}, this Hunt process possesses jointly continuous transition density $p^I_{\nu}(t, x, y)$, i,e,
\[
\bE_x [  u(X_t) ] = \int_I p^I_{\nu}(t, x, y)u(y)\nu(dy)
\]
for any bounded non-negative measurable function $u$ on $I$ and $x \in I$. By \cite[Theorems~15.6 and 15.13]{Ki16}, we have the following estimates of $p^I_{\nu}$ including ${\rm NLHK} (\phi)$. Let
\begin{equation}\label{e:alpha}
\alpha:= \tau +1 = \frac{\log (10/3)  }{\log 2} \in (1, 2).
\end{equation}
 
\thm\label{T:3.1}
There exist $c_{\ref{TSS.eq30}}, c_{\ref{TSS.eq40}}$ and $\e_{\ref{TSS.eq40}} > 0$ such that
\begin{equation}\label{TSS.eq30}
p^I_{\nu}(t, x, x) \le c_{\ref{TSS.eq30}}t^{-1/\alpha}
\end{equation}
and 
\begin{equation}\label{TSS.eq40}
 p^I_{\nu}(t, x, y) \geq c_{\ref{TSS.eq40}}t^{-1/\alpha} 
\end{equation}
for any $t\in (0, 1]$, $x, y \in I$ with $|x - y| \leq \e_{\ref{TSS.eq40}}t^{1/\alpha}$. In particular, we have
\begin{equation}\label{TSS.eq50}
p^I_{\nu}(t, x, x) \asymp t^{-1/\alpha}
\end{equation}
 for any $t \in (0, 1]$ and $x \in I$. Moreover, we have
\begin{equation}\label{TSS.eq60}
\bE_x [ \tau_{B_{*}(x, r)}] \asymp r^{\alpha}
\end{equation}
for any $r \in (0, 1]$ and $x \in I$, where $B_*(x, r) = \{y|\, y \in I, |x - y| < r\}$.
\endthm

To apply Theorem~\ref{GT.thm10} in the previous section to derive off-diagonal lower bound estimates of the transition density,  
we need to investigate the jump kernel $J_*(x, y)$ of $(\E|_I, \F|_I)$. In fact, the exact form of $J_*(x, y)$ was recently derived in \cite{KiKTaka}. To introduce the exact expression, we need several preparations. Let 
\[
\SS = \{0, 1\}^{\BbN} = \{i_1i_2 \cdots: \  i_j \in \{0, 1\}\,\, \text{for any $j \in \BbN$}\}.
\]
The set $\SS$   can be identified with the standard 
Cantor set. For $\omega = \omega_1\omega_2\ldots \in \SS$ and $\xi = \xi_1\xi_2\ldots \in \SS$, define $n_*(\omega, \xi)$ and $d_{\SS}(\omega, \xi)$ as
\[
n_*(\omega, \xi) = \min\{j| \omega_j \neq \xi_j\} - 1
\]
and 
\[
d_{\SS}(\omega, \xi) = \begin{cases}
2^{-n_*(\omega, \xi)} &\text{if $\omega \neq \xi$,}\\
0 &\text{if $\omega = \xi$.}
\end{cases}
\]
Then $d_{\SS}$ is an ultrametric on $\SS$, where the prefix ``ultra'' means that a stronger version of the triangle inequality
\[
d_{\SS}(\omega, \xi) \le \max\{d_{\SS}(\omega, \rho), d_{\SS}(\rho, \xi)\}
\]
holds. Define a map $\pi: \SS \to I$ as
\[
\pi(i_1i_2\ldots) = \sum_{j \ge 1} \frac{i_j}{2^j}.
\]
The sequence $\pi^{-1}(x)$ is the binary expression of $x \in I$. In fact, $\pi$ is injective except at the countable set
  consisting of sequences that eventually having all 0's or all 1's:  
  \[
\{i_1\cdots{i_n}0111\cdots, \  i_1\cdots{i_n}1000\cdots  : \  n \ge 0, i_1, \ldots, i_n \in \{0, 1\}\}.
\]
Note that
\begin{equation}\label{TSS.eq100}
|\pi(\omega) - \pi(\xi)| \le d_{\SS}(\omega, \xi)
\end{equation}
for any $\omega, \xi \in \SS$.

\definition\label{TSS.def100}
Define
\[
J_*(x, y) = \frac{35}{16}\Big(\frac{14}{17}\Big(\frac {20}3\Big)^{n_*(\pi^{-1}(x),\pi^{-1}(y))} + \frac{3}{17}\Big)
\]
for any $x, y \in I$. To be precise, if either $\pi^{-1}(x)$ or $\pi^{-1}(y)$ contains more than one point, we set
\[
n_*(\pi^{-1}(x), \pi^{-1}(y)) = \max\{n_*(\omega, \xi) : \   \omega \in \pi^{-1}(x), \xi \in \pi^{-1}(y)\}.
\]
\enddefinition

\medskip

\thm{\rm(\cite[Corollary~6.2]{KiKTaka})}\label{TSS.thm10}
\[
\F|_I = \bigg\{u: u \in C(I, d_*), \int_I J_*(x, y)(u(x) - u(y))^2\nu(dx)\nu(dy) < \infty\bigg\},
\]
where $d_*$ is the Euclidean metric on $I$, and
\[
\E|_I(u, v) = \int_I J_*(x, y)(u(x) - u(y))(v(x) - v(y))\nu(dx)\nu(dy)
\]
for any $u, v \in \F_I$. 
\endthm

Using $d_{\SS}$, we have
\begin{equation}\label{e:3.27} 
J_*(x, y) = \frac{35}{16}\Big(\frac{14}{17}\frac {1}{d_{\SS}(\pi^{-1}(x), \pi^{-1}(y))^{1+\alpha}} + \frac{3}{17}\Big)
\end{equation} 
 Hence by \eqref{TSS.eq100}, there exists $c_{\ref{e:3.33}} > 0$ such that
\begin{equation}\label{e:3.33}
J_*(x, y) \le \frac{c_{\ref{e:3.33}}}{|x - y|^{\a + 1}}
\end{equation}
for any $x, y \in I$. 

\thm\label{TSS.thm20}
There exists $c_{\ref{TSS.eq70}} > 0$ such that 
\begin{equation}\label{TSS.eq70}
p^I_{\nu}(t, x, y) \le c_{\ref{TSS.eq70}}\min\bigg\{t^{-1/\alpha }, \, \frac{t}{|x - y|^{1+ \alpha}}\bigg\}
\end{equation}
for any $x, y \in I$ and $t \in (0, 1]$.
\endthm

\def\wB{\widetilde{B}}
\def\wF{\widetilde{\mathcal{F}}}
\def\wl{\widetilde{\lambda}}

\demo
The inequality \eqref{e:3.33} shows that $J_{\phi, \le}$ holds with $\phi (r) :=r^\alpha$. Since $\a < 2$ by \eqref{e:alpha}, ${\rm CSJ} (\phi)$ holds by  Remark~\ref{R:2.11}.\par
Next we show ${\rm FK}(\phi)$ by comparing $J_*$ with $J^{(\infty)}$, which is the counterpart of $J_*$ in the case of the infinite Sierpinski gasket and will appear in Section~\ref{EIC}. By \eqref{IFE.eq200}, 
\begin{equation}\label{TSS.eq300}
J_*(x, y) = J^{(\infty)}(x, y) + \frac{35}{16}\cdot\frac{3}{17}.
\end{equation}
for any $(x, y) \in I^2$ with $x \neq y$. Define
\[
B(x, r) = \{y| |x - y| < r, y \in I\} \quad\text{and}\quad \wB(x, r) = \{y| |x - y| < r. y \in [0, \infty)\}
\]
for $x \in I$. Moreover, we denote $\F_D$ and $\lambda_1(D)$ in Definition~\ref{GT.def10} for $(\E_{\BbR_+}^{(\infty)}, \F_{\BbR_+}^{(\infty)})$ by $\wF_D$ and $\wl(D)_1$ respectively, while $\F_D$ and $\lambda_1(D)$ represents those for $(\E|_I. \F|_I)$. As is mentioned in Section~\ref{EIC}, it is shown in \cite{Ki25} that $\rm UHK(\phi)$ holds for $(\E_{\BbR_+}^{(\infty)}, \F_{\BbR_+}^{(\infty)})$. Hence Theorem~\ref{T:2.12} implies ${\rm FK}(\phi)$ for $(\E_{\BbR_+}^{(\infty)}, \F_{\BbR_+}^{(\infty)})$. More precisely, there exist $c > 0$ and $\nu > 0$ such that if $D$ is an open subset of $\wB(x, r)$ for $x \in [0, \infty)$ and $r > 0$, then
\begin{equation}\label{TSS.eq310}
\wl_1(D) \ge \frac{c}{\phi(r)}\Big(\frac{V(x, r)}{\nu(D)}\Big)^{\theta}.
\end{equation}
Let $\sigma = \frac 12$, $r \in (0, \sigma)$, and let $D$ be an open subset of $B(x, r)$. Suppose $1 \notin B(x, r)$. Then $D$ is an open subset of $[0, \infty)$ as well, and $\F_D = \wF_D$.  Moreover, by \eqref{TSS.eq300}, we see that $\lambda_1(D) \ge \wl_1(D)$. Therefore, by \eqref{TSS.eq300}, 
\begin{equation}\label{TSS.eq320}
\lambda_1(D) \ge \frac{c}{\phi(r)}\Big(\frac{V(x, r)}{\nu(D)}\Big)^{\theta}.
\end{equation}
Next suppose $1 \in B(x, r)$. Define $\omega(x) = 1 - x$ for $x \in [0, 1]$. Then $\omega(B(x, r)) = B(\omega(x), r)$. Since $r \in (0, \sigma)$, it follows that $1 \notin B(\omega(x), r)$. Furthermore $\lambda_1(D) = \lambda_1(\omega(D))$, $V(\omega(x), r) = V(x, r)$ and $\nu(D) = \nu(\omega(D)$. Hence \eqref{TSS.eq320} holds for this case as well. Thus we have shown ${\rm FK(\phi)}$ for $(\E|_I, \F|_I)$. \par
Now by Theorem~\ref{T:2.12} and Remark~\ref{R:2.10}-(ii), we have $\rm UHK(\phi)$ for $(\E|_I, \F|_I)$.
\enddemo

Since $J_* (x, y)$ is not comparable with $1/|x - y|^{1+\alpha}$, a comparable lower bound estimate does not hold for $p^I_{\nu}(t, x, y)$. Nevertheless, applying Theorem~\ref{GT.thm10}, we are going to show a lower off-diagonal estimate in Theorem~\ref{TSS.thm40}. To start with, for $a\in \bR$, define $\vp_a: \BbR^2 \to \BbR^2$ by
\[
\vp_a(x, y) = \left(\tfrac 12(x - a) + a, \,  \tfrac 12(y - a) + a\right).
\]
Set 
\[
W_m = \{0, 1\}^m = \{i_1\ldots{i_m} : \  i_j \in \{0, 1\}\,\,\text{for any $i = 1, \ldots, m$}\}
\]
for $m \ge 1$, $W_0 = \{\emptyset\}$, $W_* = \cup_{m \ge 0} W_m$, $\vp_{\emptyset} = \text{the identity}$ and
\[
\vp_{i_1\ldots{i_m}} = \vp_{i_1}{\circ}\cdots{\circ}\vp_{i_m}
\]
for $i_1\ldots{i_m} \in W_*$. Moreover, define $|w|$ for $w \in W_*$ by $|w| = m$ if $w \in W_m$.\par
Under these preparations, we have a division of $\sd{I^2}{\!\rm diag}$ according to values of $J_*(x, y)$ shown in Figure~\ref{Fig1}.

\begin{figure}
\centering
\includegraphics[width=\linewidth]{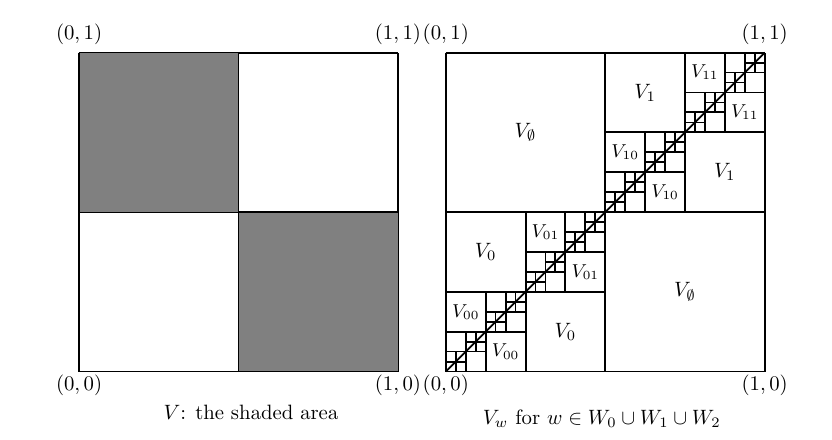}
\caption{Division of $\sd{I^2}{\{(x, x)|x \in I\}}$ in \eqref{TSS.eq90}}\label{Fig1}
\end{figure}

Namely, let
\[
V = \bigg(\Big(\frac 12, 1\Big] \times \Big[0, \frac 12\Big)\bigg) \bigcup \bigg(\Big[0, \frac 12\Big) \times \Big(\frac 12, 1\Big]\bigg)
\quad \text{and} \quad V_w = \vp_w(V)
\]
for $w \in W_*$. Then
\begin{equation}\label{TSS.eq90} 
\sd{I^2}{\!\rm diag}
= \bigsqcup_{w \in W_*} V_w,
\end{equation}
where the symbol $\sqcup$ represents a disjoint union, and
\begin{equation}\label{TSS.eq200}
J_*(x, y) = \frac{35}{16}\Big(\frac{14}{17}\Big(\frac {20}3\Big)^{|w|} + \frac{3}{17}\Big)
\end{equation}
for any $(x, y) \in V_w$ and $w \in W_*$. In view of \eqref{TSS.eq200}, we see that the UJS condition holds for the kernel $J_*$.

\prop\label{TSS.prop10}
The jump kernel $J_*(x, y)$ satisfies the following UJS condition \eqref{GT.eq20}. More precisely,
for every $x\neq y \in I$, 
\begin{equation}\label{e:3.14} 
J_*(x, y) \leq \frac 2{r} \int_{B(x, r)} J_* (z, y)\nu(dz)
\end{equation}
for any $x, y \in I$ and $r > 0$ with $0 < r < |x - y|/2$.
\endprop

Note that $r \le \mu(B(x, r)) \le 2r$ for any $x \in I$ and $r \in (0, 1]$.

\demo
Assume that $(x, y) \in V_w$ for some $w \in W_*$. Then
\[
V_w = \big((a, a + L] \times [a - L, a)\big) \cup \big([a - L, a) \times (a, a + L]\big),
\]
where $a = \vp_w\big((\frac 12, \frac 12)\big)$ and $L = 2^{-|w| - 1}$. Note that $|x - y| < 2L$. Let $I_- = \{(z, y)| z \in (x - r/2)\}$ and $I_+ = \{(z, y)| z \in (x, x + r/2)\}$. If $r < |x - y|/2$, then $r < L$ and hence either $I_+ \subseteq V_w$ or $I_+ \subseteq V_w$. Say $I_- \subseteq V_w$. Since $J_*(x, y) = J(z, y)$ for any $(z, y) \in I_-$, it follows that
\[
\frac r2J_*(x, y) = \int_{I_-} J(z, y)\nu(dz) \le \int_{B(x, r)} J_*(z, y)\nu(dz).
\]
Thus we verify \eqref{e:3.14}. The argument is entirely the same in the case $I_+ \subseteq V_w$.
\enddemo

Now that ${\rm UHK}(\phi)$, ${\rm NLHK} (\phi)$ and {\rm UJS} have obtained, Corollary~\ref{GT.cor200} implies the following lower estimate of $p^I_{\nu}(t, x, y)$.

\thm\label{TSS.thm40}
There exists are positive constants $c_{\ref{HKE}} > 0$ and $\e_{\ref{HKE}} \in (0, \e_{\ref{TSS.eq40}}/2)$ such that
\begin{equation}\label{HKE}
p^I_{\nu}(t, x, y) \ge 
\begin{cases}
c_{\ref{TSS.eq40}}t^{-1/\alpha}\quad&\text{if $|x - y| \le \e_{\ref{TSS.eq40}}t^{1/\alpha}$},\\
c_{\ref{HKE}}tJ_*(x, y, \e_{\ref{HKE}}t^{1/\a})\quad&\text{otherwise},
\end{cases}
\end{equation}
for any $x, y \in I$ and $t \in (0, 1]$, where
\[
J_*(x, y, r) = \frac 1{\nu(B(x, r))\nu(B(y, r))}\int_{B(x, r)}\int_{B(y, r)} J_*(u, v)\nu(du)\nu(dv).
\]
\endthm

Due to UJS, we may replace $J_*(x, y, \e_{\ref{HKE}}t^{1/\a})$ in \eqref{HKE} with $J_*(x, y)$ by Corollary~\ref{GT.cor100}, However, in this case, this does not make the estimate sharper because the former is continuous with respect to $x$ and $y$ but the latter is not. \par
Finally, we consider the consistency between near-diagonal and off-diagonal estimates in \eqref{HKE}. This problem has been raised at the end to the last section.  In the current example, the consistency fails because of the next proposition.

\prop\label{TSS.prop20}
For any $\rho \in (0, \frac 12)$, 
\[
\inf_{(x, y) \in \sd{I^2}{\!\rm diag}} V(x, |x - y|)\phi(|x - y|)J_*(x, y, \rho{|x - y|}) = 0
\]
\endprop
\demo
Let $x \in (\frac 12, \frac 34)$ and $y = 1 - x$. Then $|x - y| = 2x - 1$. Since $\rho \in (0, \frac 12)$, we see that $B(x, \rho|x - y|) \times B(y, \rho|x - y|) \subseteq V = V_{\phi}$. Therefore, $J_*(u, v) = J_*(x, y) = \frac{35}{16}$ for any $(u, v) \in B(x, \rho|x - y|) \times B(y, \rho|x - y|)$. Thus
\[
V(x, |x - y|)\phi(|x - y|)J_*(x, y, \rho|x - y|) = \frac {35}{16}\times 2|x - y|^{\a + 1} = \frac{35}8|2x - 1|^{\a + 1}.
\]
Taking $x \downarrow \frac 12$, we see the desired result.
\enddemo

 \begin{remark}\label{R:4.8} \rm   
 This remark is concerned with a closely related but different trace process than $X$ studied in this section; 
 namely, the trace of the reflected process of the part process
 of  the Brownian motion on the planar Sierpinski gasket $K$ killed upon hitting the bottom line segment $I$. 
 
 \medskip
 
 Let $D=K\setminus I$, the planar Sierpinski gasket with bottom line segment removed. Recall that $d_*$ is the Euclidean distance restricted to $K$.
Denote by $\rho_D$  the geodesic distance in $D$, that is,  
\[\rho_D(x,y)=\inf\{\operatorname{length}(\gamma):\,\gamma\hbox{ is a rectifiable path in $D$ connecting }x,y\}\hbox{ for }x,y\in D,\]
where $\operatorname{length}(\gamma)$ is the length of a continuous rectifiable curve $\gamma$ in $\R^2$ metered with Euclidean distance. 
Denote by $(D^*,\rho_D)$ the completion of $(D, \rho_D)$, and let $\partial D^*=D^*\setminus D$.
 We can identify $\partial D^*$ with the ``Cantor set" $\{1,2\}^{\mathbb N}$ (see \cite[Section 9.1]{CC}, which 
 is the same as the standard Cantor set in $[0, 1]$ except at the countably many points of the form 
 $$
 \left\{\sum_{k=1}^{n-1} \frac{a_k}{ 3^k} + \frac{1}{3^n}: \ n\geq 1,
     a_k\in \{ 0, 1\}  \hbox{ for } 0\leq k\leq n \right\}.
 $$
 each of which splits into the left and right points. Thus $\partial D^*$ can be identified with the unit interval $I$ except at the countably many dyadic points
 $\sN$ which has zero Lebesgue measure. Using the induced metric $\rho_D$ on $\partial D^*$ and the above identification of $\partial D^*$, 
 we can express the bounds of the jump kernel $J_*$ of Definition \ref{TSS.def100} in a more compact form:
 \begin{equation}\label{e:4.19} 
 J_*(x, y) \asymp  \frac{1}{ \rho_D(x,y)^{ \log(20/3) /\log2}}
 \quad \hbox{for } x, y\in I\setminus \sN
 \end{equation}

 Let $X^D$ be the part process of the Brownian motion $X$  on the Sierpinski gasket $K$ killed upon hitting its bottom line $I$.
 We point out that the trace of $X$ on $I$ is different from the trace of the reflected Brownian motion $\bar X$ of $X^D$ on $\partial D^*$. 
 The latter has been studied in \cite[Section 9.1]{CC}, where the jump kernel is shown to have the same estimates as in \eqref{e:4.19}
 and its two-sided transition estimates are given. For the trace process of $\bar X$ on $\partial D^*$, each point in $\sN$ is split into two points:
 the left and right point and the distance between them is positive under the metric $\rho_D$. For the trace process of $X$ on $I$, there is 
 no splitting of the points in $\sN$.  \qed
 \end{remark}
 
\setcounter{equation}{0}
\section{Example: SG trace continued}\label{S:4}
In this section, we are going to show a way to modify $J_*(x, y)$ so that we have an equivalent resistance form on the interval as $\E_I$.\par

\thm\label{T:4.1}
Set
$$
D_s = \bigg(\Big[\frac 12, 1 - \frac s2\Big) \times \Big(\frac s2, \frac 12\Big]\bigg) \bigcup \bigg(\Big(\frac s2, \frac 12\Big] \times \Big[\frac 12, 1 - \frac s2\Big)\bigg).
$$
for $s \in (0, 1)$ and set
\begin{equation}\label{e:Ns}
N_s = \bigcup_{w \in W_*} \vp_w(D_s).
\end{equation}
Assume that $J(x, y)$ satisfy $J(x, y) = J(y, x)$ for any $(x, y) \in [0, 1]^2$,
\begin{equation}\label{EF.asu10}
0 \le J(x, y) \le J_*(x, y)
\end{equation}
if $(x, y) \in N_s$ and
\begin{equation}\label{EF.asu20}
J(x, y) = J_*(x, y)
\end{equation}
otherwise. Define
\[
\F_J = \bigg\{u: u \in L^2(I, \nu), \int_{I^2} (u(x) - u(y))^2J(x, y)\nu(dx)\nu(dy) < +\infty\bigg\}
\]
and
\[
\E_J(u, v) = \int_{I^2} (u(x) - u(y))(v(x) - v(y))J(x, y)\nu(dx)\nu(dy)
\]
for $u, v \in \F_J$. Then $\F_J = \F|_I$ and there exist $c_{\ref{EF.eq10}.1}, c_{\ref{EF.eq10}.2} > 0$ such that
\begin{equation}\label{EF.eq10}
c_{\ref{EF.eq10}.1}\E|_I(u, u) \le \E_J(u, u) \le c_{\ref{EF.eq10}.2}\E|_I(u, u).
\end{equation}
for any $u \in \F|_I$.
\endthm

A proof of this theorem will be given at the end of this section. Meanwhile, we discuss the consequences from the above theorem. By \eqref{EF.eq10},  $(\E_J, \F_J)$ is a resistance form and 
\begin{equation}\label{EF.eq15}
(c_{\ref{EF.eq10}.2})^{-1}R(x, y) \le R_J(x, y) \le (c_{\ref{EF.eq10}.1})^{-1}R(x, y)
\end{equation}
for any $x, y \in I$, where $R_J(x, y)$ is the resistance metric associated with $(\E_J, \F_J)$. Hence using \cite[Theorems~15.6 and 15.13]{Ki16}, we see that the counterpart of \eqref{TSS.eq30}-\eqref{TSS.eq50} holds for the process associated with the Dirichlet from $(\E_J, \F_J)$ on $L^2(I, \nu)$. Moreover, since
\[
0 \le J(x, y) \le J_*(x, y) \le \frac{c_*}{|x - y|^{1 + \alpha}}
\]
for any $x, y \in I$, we have the counterpart of \eqref{TSS.eq70} as well. As we have ${\rm UHK}(\phi)$ and ${\rm NLHK} (\phi)$ with $\phi(r) = r^{\a}$, Theorem~\ref{GT.thm10} immediately yields the following lower off-diagonal estimate of $p_J(t, x, y)$.

\thm\label{EF.thm10}
Let $p_J(t, x, y)$ be the jointly continuous transition density associated with the Dirichlet form $(\E_J, \F_J)$ on $L^2(I, \nu)$. There exists positive constants $t_0$, $c_{\ref{EF.eq20}.1}, c_{\ref{EF.eq20}.1}$, $\e_{\ref{EF.eq20}.1}$ and $\e_{\ref{EF.eq20}.2} \in (0, \frac 12\e_{\ref{EF.eq20}.1})$ such that
\begin{equation}\label{EF.eq20}
p_J(t, x, y) \ge \begin{cases}
c_{\ref{EF.eq20}.1}t^{-1/\alpha}\quad&\text{if $|x - y| \le \e_{\ref{EF.eq20}.1}t^{1/\alpha}$},\\
c_{\ref{EF.eq20}.1}tJ(x, y, \e_{\ref{EF.eq20}.2}t^{1/\a})\quad&\text{otherwise},
\end{cases}
\end{equation}
for any $(x, y) \in \sd{I^2}{\!{\rm diag}}$ and $t \in (0, t_0]$.
\endthm

In fact, exactly the same arguments as in the proof of Proposition~\ref{TSS.prop10} shows the following partial UJS property.

\prop\label{EF.prop10}
\begin{equation}\label{EF.eq30} 
J(x, y) \leq \frac 2{r} \int_{B(x, r)} J (z, y) \nu(dz)
\end{equation}
for any $(x, y) \in \sd{I^2}{N_s}$ and $r > 0$ with $0 < r < |x - y|/2$, where $N_s$ is the set defined in \eqref{e:Ns}.
\endprop

By This proposition, it is possible to replace $J(x, y, \e_{\ref{EF.eq20}.2}t^{1/\a})$ in the off-diagonal part of \eqref{EF.eq20} by $J(x, y)$ if $(x, y) \in \sd{I^2}{N_s}$.\par
The UJS property for whole the space depends on the nature of $J$ on $N_s$. For example, if $J(x, y) \equiv 0$ on $N_s$, then \eqref{EF.eq30} holds for any $(x, y) \in I^2$ with $x \neq y$ so that UJS holds. On the other hand, the following proposition gives an example where UJS on the whole space fails.

\prop\label{EF.prop20}
Let $h = \frac 12, a_n = h + h^{n + 2}$, $b_n = h - h^{n + 2}$ and $\rho_n = h^{n^2 + 4}$ for $n \ge 1$.  Define $J: \sd{I^2}{\{(x, x)| x \in I\}} \to [0, \infty)$ by 
\[
J = \1_{\sd{I^2}{N_h}}J_* + \sum_{n \ge 1} \big(\1_{B(a_n, \rho_n) \times B(b_n, \rho_n)} + \1_{B(b_n, \rho_n) \times B(a_n, \rho_n)}\big).
\]
Then $J$ satisfies the assumptions \eqref{EF.asu10} and \eqref{EF.asu20} and
\[
\inf\Big\{\frac 1{J(x, y)r}\int_{B(x, r)} J(z, y)\nu(z) :  \  (x, y) \in I^2, x \neq y, 0 < r <  |x - y|/2\Big\} = 0.
\]
In particular, $\rm UJS(\phi)$ does not hold.
\endprop

\demo
Let $r_n = h^{n + 4}$. Then $B(a_n, r_n) \cap B(a_m, r_m) = \emptyset$ and $B(b_n, r_n) \cap B(b_m, r_m) = \emptyset$ if $n \neq m$. Moreover, 
\[
\big(B(a_n, r_n) \times B(b_n, r_n)\big) \cup \big(B(b_n, r_n) \times B(a_n, r_n)\big) \subseteq D_h \subseteq N_h.
\]
for any $n \ge 1$. Hence it follows that $J \equiv 0$ on $\sd{N_h}{D_h}$ and 
\[
J(x, y) \le 1 \le \frac{35}{16} \equiv J_*|_{D_h}
\]
for any $(x, y) \in D_h$. Thus we see \eqref{EF.asu10} and \eqref{EF.asu20}. Note that $J(a_n, b_n) = 1$. Hence
\[
\frac 1{J(a_n, b_n)r_n}\int_{B(a_n, r_n)} J(z, b_n) = 2\frac{h^{n^2 + 4}}{h^{n + 4}} = 2h^{n^2 - n}.
\]
Letting $n \to \infty$, we see the desired statement.
\enddemo

\begin{figure}
\centering
\includegraphics[width=\linewidth]{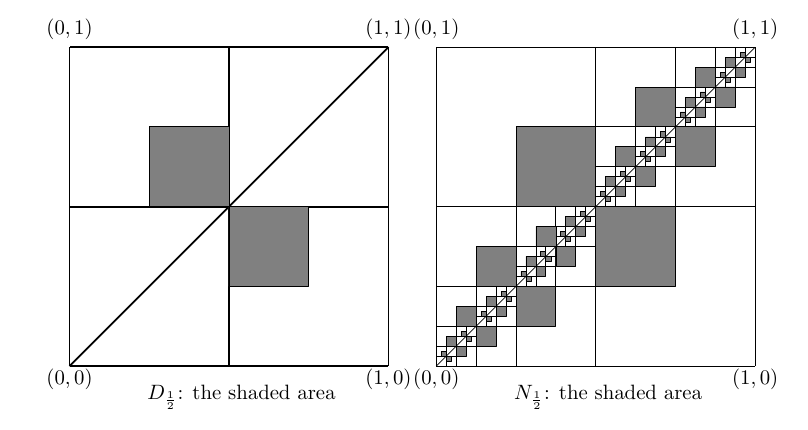}
\caption{$D_{s}$ and $N_{s}$} 
\label{Fig3}
\end{figure}

Now we begin to prove Theorem~\ref{T:4.1}. An essential part of a proof is the next lemma.

\lemma\label{EF.lemma10}
Let $a \in \BbR$, $h > 0$, $s \in (0, 1)$. Define $I_1 = [a - h, a - (1 - s)h], I_2 = [a - (1 - s)h, a], I_3 = [a, a + (1 - s)h]$ and $I_4 = [a + (1 - s)h, a + h]$. Set $D_{i, j} = I_i \times I_j$ for $(i, j) = (3, 1), (3, 2), (4, 1), (4, 2)$. There exists $\e > 0$ such that 
\[
\e\int_{D_{3, 2}} (u(x) - u(y))^2\nu(dx)\nu(dy)  \le \int_{D_{3, 1} \cup D_{4, 1} \cup D_{4, 2}} (u(x) - u(y))^2\nu(dx)\nu(dy)
\]
for any $a \in \BbR$, $h > 0$ and any measurable function $u: [a - h, a + h] \to \BbR$.
\endlemma

\demo
Set $X_i = \int_{I_i} u(x)^2dx$ and $Y_i = \int_{I_i} u(x)dx$. Then
\[
\int_{D_{i, j}} (u(x) - u(y))^2\nu(dx)\nu(dy) = |I_j|X_i + |I_i|X_j - 2Y_iY_j,
\]
where $|I_i|$ is the length of $I_i$. Note that replacing $u$ by $u - |I_1|^{-1}Y_1$, we can make $Y_1 = 0$ without changing the values of $\int_{D_{i, j}}(u(x) - u(y))^2\nu(dx)\nu(dy)$. Hence, letting $\e$ small enough, we see that
\begin{multline}\label{EF.eq50}
\int_{D_{3, 1} \cup D_{4, 1} \cup D_{4, 2}} (u(x) - u(y))^2\nu(dx)\nu(dy) - \e\int_{D_{3, 2}}(u(x) - u(y))^2\nu(dx)\nu(dy)\\
 = (|I_1| - \e|I_2|)X_3 + hX_1 +hX_4 + (|I_4| - \e|I_3|)X_2 - 2Y_2Y_4 + 2\e{Y_2Y_3} \\
 \ge   \frac{s - \e(1 - s)}{1 - s}(Y_3)^2 + \frac 1s(Y_1)^2 + \frac 1s(Y_4)^2 + \frac{s - \e(1 - s)}{1- s}(Y_2)^2 - 2Y_2Y_4 + 2\e{Y_2Y_3}
\end{multline}
If $\e = 0$ in the last line of \eqref{EF.eq50}, we have
\begin{align*}
 &\frac{s}{1 - s}(Y_3)^2 + \frac 1s(Y_1)^2 +  \frac 1s(Y_4)^2 + \frac{s}{1 - s}(Y_2)^2 - 2Y_2Y_4\\
 = &\frac{s}{1 - s}(Y_3)^2 + \frac 1s(Y_1)^2 + (Y_4)^2 + \frac 1{s(1 - s)}((1 - s)Y_4 - sY_2)^2
 \end{align*}
 This is a positive definite quadratic form of $Y_1, Y_2, Y_3, Y_4$ and hence so is the last line of \eqref{EF.eq50} for a small $\e$ which does not depend on $a$ nor $h$.
 \enddemo
 
 \demo[Proof of Theorem~\ref{T:4.1}]
 Define
 \[
 \uJ(x, y) = \begin{cases}
 0  \quad&\,\,\text{if $(x, y) \in N_s$},\\
 J_*(x, y) \quad&\,\,\text{otherwise}.
 \end{cases}
 \]
 It is enough to show the theorem in the case $J = \uJ$. By Lemma~\ref{EF.lemma10}, 
 \begin{multline*}
 \int_{\vp_w(\sd{D_0}{D_s})} (u(x) - u(y))^2\nu(dx)\nu(dy) \le  \int_{\vp_w(D_0)} (u(x) - u(y))^2\nu(dx)\nu(dy)\\
\le \int_{\vp_w(\sd{D_0}{D_s})} (u(x) - u(y))^2\nu(dx)\nu(dy) + \int_{\vp_w(D_s)} (u(x) - u(y))^2\nu(dx)\nu(dy)\\ 
\le \Big(1 + \frac 1{\e}\Big)\int_{\vp_w(\sd{D_0}{D_s})} (u(x) - u(y))^2\nu(dx)\nu(dy)
 \end{multline*}
 for any $w \in W_*$ and any measurable function $u: D_0 \to \BbR$. This implies
 \[
 \E^{(\uJ)}(u, u) \le \E_I(u, u) \le \Big(1 + \frac 1{\e}\Big)\E^{(\uJ)}(u, u)
 \]
 for any measurable function $u: I \to \BbR$.
 \enddemo
 
\setcounter{equation}{0}
\section{Example: an infinite counterpart of Section~\ref{S:3}}\label{EIC}
\def\uS{\underline{S}}
\def\oS{\overline{S}}

In this section, we present an example of a jump process/Dirichlet form on the half line $\BbR_+ = [0, \infty)$, where Theorem~\ref{GT.thm10} is applicable. It is an infinite version to the example in Section~\ref{S:3}, i.e. consider a natural infinite Sierpinski gasket, extend the standard resistance form $(\E, \F)$ on the Sierpinski gasket $K$ to it, and take the trace of the extended resistance form on the infinite counterpart of $I = [0, 1]$, which is the half-line $\BbR_+$. Throughout this section, $K$ is the Sierpinski gasket and $(\E, \F)$ is the standard resistance form on the Sierpinski gasket $K$ given in Section~\ref{S:3}.

\definition\label{IFE.def00}
Define $\rho^{(m)}: \BbR^2 \to \BbR^2$ by $\rho^{(m)}(x) = 2^mx$ for $x \in \BbR^2$. Moreover, define $K^{(m)} = \rho^{(m)}(K)$, $\F^{(m)} = \{f{\circ}(\rho^{(m)})^{-1}| f \in \F\}$ and
\[
\E^{(m)}(u, v) = \Big(\frac 35\Big)^m\E(u\circ\rho^{(m)}, v\circ\rho^{(m)})
\]
for $u, v \in \F^{(m)}$. 
\enddefinition

By a scaling argument, $(\E^{(m)}, \F^{(m)})$ is shown to be a resistance form on $K^{(m)}$ and if $f \in \F^{(m + 1)}$, then $f|_{K^{(m)}} \in \F^{(m)}$ and $\E^{(m)}(f|_{K^{(m)}}, f|_{K^{(m)}}) \le \E^{(m + 1)}(f, f)$. \par
Let 
\[
K^{(\infty)} = \bigcup_{m \ge 0} K^{(m)},
\]
which is called the infinite Sierpinski gasket, which is illustrated in the left of Figure~\ref{ISG+}.
 In \cite{Ki25}, a resistance form on $K^{(\infty)}$, which is called the standard resistance form on $K^{(\infty)}$, has been constructed and shown to be a ``limit'' of the sequence $\{(\E^{(m)}, \F^{(m)})\}_{m \ge 0}$. More precisely, we have the next theorem.

\thm[\cite{Ki25}]\label{IFE.thm00}
Define
\begin{multline*}
\F^{(\infty)} = \{u: u \in C(K^{(\infty)}), u|_{K^{(m)}} \in \F^{(m)}\,\,\text{for any $m \ge 0$} \\
\text{and $\lim_{m \to \infty} \E^{(m)}(u|_{K^{(m)}}) < \infty$}\}
\end{multline*}
and
\[
\E^{(\infty)}(u, v) = \lim_{m \to \infty} \E^{(m)}(u|_{K^{(m)}}, v|_{K^{(m)}})
\]
for $u, v \in \F^{(\infty)}$. Then $(\E^{(\infty)}, \F^{(\infty)})$ is a resistance form on $K^{(\infty)}$. Moreover, let $R^{(\infty)}$ be the associated resistance metric on $K^{(\infty)}$. Then
\begin{equation}\label{IFE.eq100}
c_1|x - y|^{\tau} \le R^{(\infty)}(x, y) \le c_2|x - y|^{\tau}
\end{equation}
for any $x, y \in K^{(\infty)}$, where $\tau:= \frac{\log (5/3)}{\log 2}$ as in \eqref{e:tau}.
\endthm

The resistance form $(\E^{(\infty)}, \F^{(\infty)})$ on $K^{(\infty)}$ is called the standard resistance form on the infinite Sierpinski gasket.   The local regular Dirichlet form associated with the resistance form $(\E^{(\infty)}, \F^{(\infty)})$ is the one associated with the Brownian motion on $K^{(\infty)}$ constructed by \cite{BP}. See \cite{Ki25} for details.

\par

Our example of a Dirichlet form on $\BbR_+$ is the trace of $(\E^{(\infty)}, \F^{(\infty)})$ on $\BbR_+ \times \{0\} \subseteq K^{(\infty)}$.  As in the case of $(\E|_I, \F|_I)$, we have the following theorem by \cite[Theorem~8.4]{Ki16}.

\remark
For simplicity, we identify $\BbR_+ \times \{0\}$ with $\BbR_+$ hereafter. 
\endremark

\thm[Trace of $(\E^{(\infty)}, \F^{(\infty)})$]\label{IFE.thm10}
Define
\[
\F^{(\infty)}_{\BbR_+} = \{v : \text{ there exists $u \in \F^{(\infty)}$ such that $v = u|_{\BbR_+}$}\}
\]
and
\[
\E^{(\infty)}_{\BbR_+}(v) = \inf\{\E^{(\infty)}(u, u) : \  u \in \F^{(\infty)}, u|_{\BbR_+} = v\}
\]
for $v \in \F^{(\infty)}_{\BbR_+}$. Then $(\E^{(\infty)}_{\BbR_+}, \F^{(\infty)}_{\BbR_+})$ is a resistance form on $\BbR_+$.
\endthm

Note that the resistance metric associated with the resistance form $(\E^{(\infty)}_{\BbR_+}, \F^{(\infty)}_{\BbR_+})$ is the restriction of $R^{(\infty)}$ to $\BbR_+$. Consequently, the inequality \eqref{IFE.eq100} holds for any $x, y \in \BbR_+$ and the resistance metric gives $\BbR_+$ the same topology as the Euclidean metric.\par
By \cite[Theorem~9.4]{Ki16}, we obtain a regular Dirichlet form from the resistance form $(\E^{(\infty)}_{\BbR_+}, \F^{(\infty)}_{\BbR_+})$ as follows.

\thm\label{IFE.thm20}
Let $\nu_*$ be the restriction of the $1$-dimensional Lebesgue measure to $\BbR_+$. Define
\[
\E^{(\infty)}_{\BbR_+, *}(u, v) = \E^{(\infty)}_{\BbR_+}(u, v) + \int_{\BbR_+} u(x)v(x)\nu_*(dx)
\]
for any $u, v \in \F^{(\infty)}_{\BbR_+} \cap L^2(\BbR_+, \nu_*)$. Define $\D^{(\infty)}_{\BbR_+}$ be the completion of $\F^{(\infty)}_{\BbR_+} \cap C_0(\BbR_+)$ with respect to the inner-product $\E^{(\infty}_{\BbR_+, *}$. Then $(\E^{(\infty)}_{\BbR_+}, \D^{(\infty)}_{\BbR_+})$ is a regular Dirichlet form on $L^2(\BbR_+, \nu_*)$.
\endthm

To describe an expression of $\E^{(\infty)}_{\BbR_+}$ by a jump kernel, we need the following decomposition of $(\BbR_+)^2$.
 
\definition\label{IFE.def10}
For $i \in \BbN$ and $n \in \BbZ$. Define
\[
\uS_{n, i} = \bigg(\frac{2i - 1}{2^n}, \frac{2i}{2^n}\bigg] \times \bigg[\frac{2i - 2}{2^n}, \frac{2i - 1}{2^n}\bigg),
\]
\[
\oS_{n, i} = \bigg[\frac{2i - 2}{2^n}, \frac{2i - 1}{2^n}\bigg) \times \bigg(\frac{2i - 1}{2^n}, \frac{2i}{2^n}\bigg]
\]
and
\[
S_{n, i} = \uS_{n, i} \cup \oS_{n, i}.
\]
\enddefinition

\begin{figure}
\centering
\includegraphics[width = 300pt]{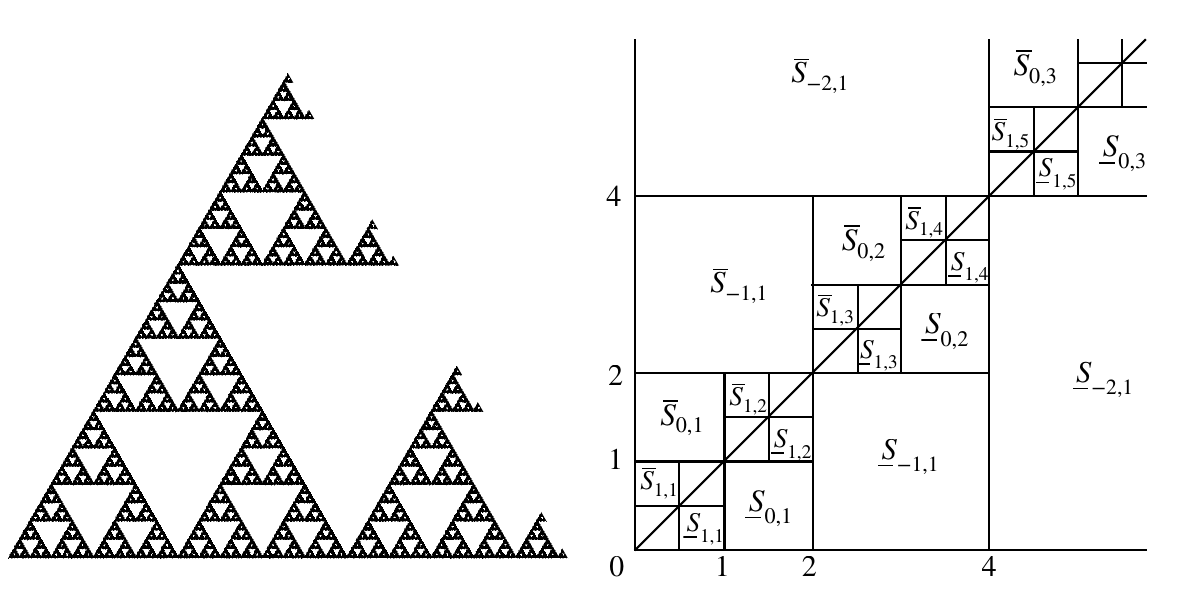}
\caption{The infinite Sierpinski gasket $K^{(\infty)}$ and the partition of $(\BbR^+)^2$}\label{ISG+}
\end{figure}

Then $S_{n, i} \cap S_{m, j} = \emptyset$ if and only if $(n, i) \neq (m, j)$ and
\[
\bigsqcup_{i \in \BbN, n \in \BbZ} S_{n, i} = (\BbR_+)^2 \backslash \{(x, x)| x \in (0, \infty)\}
\]
This decomposition of $[0, \infty)^2$, illustrated in the right of Figure~\ref{ISG+}, 
is a natural extension of the decomposition \eqref{TSS.eq90} of $[0, 1]^2$.

The following expression of $\E^{(\infty)}_{\BbR_+}$ has been obtained in \cite{Ki25}.

\thm\label{IFE.thm30}
Define $J^{(\infty)}: \sd{[0, \infty)^2}{\{(x, x)| x \in [0, \infty)\}} \to [0, \infty)$ by
\begin{equation}\label{IFE.eq200}
J^{(\infty)}(x, y) = \frac{35}{16}\cdot\frac{14}{17}\Big(\frac {20}3\Big)^{n - 1}
\end{equation}
for $n \in \BbZ$ and $(x, y) \in \bigcup_{i \in \BbN}S_{n, i}$. Then
\begin{equation}\label{IFE.eq300}
\E^{(\infty)}_{\BbR_+}(u, v) = \int_{[0, \infty)^2}(u(x) - u(y))(v(x) - v(y))J^{(\infty)}(x, y)\nu_*(dx)\nu_*(dy)
\end{equation}
for any $u, v \in \D^{(\infty)}_{\BbR_+}$.
\endthm

\remark
\rm In the definition of $J^{(\infty)}$ above, we leave the produce of two fractions not reduced so that one can easily compare it with the expression of its counterpart $J_*$ on the Sierpinski gasket in \eqref{TSS.eq200}.
\endremark 
 
Now we are in a positive to give estimates of the transition density associated with the Dirichlet form $(\E^{(\infty)}_{\BbR_+}, \D^{(\infty)}_{\BbR_+})$. The expression \eqref{IFE.eq300} enable us to apply Theorem~\ref{GT.thm10}, whose prerequisites, ${\rm UHK}(\phi)$ and ${\rm NLHK} (\phi)$ have been shown in \cite{Ki25} with $\phi(r) = r^{\alpha}$ with $\alpha:=1+\tau= \frac{\log (10/3)}{\log 2}$.

\thm\label{IFE.thm50}
There exists a jointly continuous transition density $p^{(\infty)}_{\BbR_+}(t, x, y)$ associated with the Dirichlet form $(\E^{(\infty)}_{\BbR_+}, \D^{(\infty)}_{\BbR_+})$ on $L^2(\BbR_*, \nu_*)$. Moreover, there exists $c_{\ref{IFE.eq400}} > 0$ such that
\begin{equation}\label{IFE.eq400}
p^{(\infty)}_{\BbR_+}(t, x, y) \le c_{\ref{IFE.eq400}}\min\left\{ t^{-1/\alpha} , \,  \frac{t}{|x - y|^{1+\alpha}}\right\}
\end{equation}
for every $t>0$ and  $(x, y) \in [0, \infty)^2$, and there exist positive constants $c_{\ref{IFE.eq500}}$, $\e_{\ref{IFE.eq500}.1}$ and $\e_{\ref{IFE.eq500}.2} \in (0, \e_{\ref{IFE.eq500}.1}/2)$ such that
\begin{equation}\label{IFE.eq500}
p^{(\infty)}_{\BbR_+}(t, x, y) \ge c_{\ref{IFE.eq500}}
\begin{cases}
t^{-1/\alpha}\quad&\text{if $|x - y| \le \e_{\ref{IFE.eq500}.1}t^{1/\alpha}$}\\
tJ^{(\infty)}(x, y, \e_{\ref{IFE.eq500}.2}t^{1/\a})\quad&\text{otherwise}.
\end{cases}
\end{equation}
\endthm

\demo
The existence of jointly continuous transition density is due to \cite[Theorem~10.4]{Ki16}. The estimates ${\rm UHK}(\phi)$, which is \eqref{IFE.eq400}, and ${\rm NLHK} (\phi)$, which is the upper inequality of \eqref{IFE.eq500}, have been shown in \cite{Ki25}. The rest is immediate by Theorem~\ref{GT.thm10}.
\enddemo

Additionally, exactly the same arguments as in the proof of Proposition~\ref{TSS.prop10} implies UJS condition for $J^{(\infty)}$.

\prop\label{IFE.prop10}
 The jump kernel $J^{(\infty)}(x, y)$ satisfies the following UJS condition:
 \[
 J^{(\infty)} (x, y) \leq \frac2r \int_{B(x, r)} J^{(\infty)}(z, y)\nu_*(dy)
 \] 
for any $(x, y) \in (\BbR_+)^2$ with $x \neq y$ and $r > 0$ with $0 < r < |x - y|/2$.
\endprop

However, as in the previous examples, there is no advantage of UJS in improving the off-diagonal part of \eqref{IFE.eq500}. 

\medskip

Finally the consistency between the near-diagonal and the off-diagonal estimates in \eqref{IFE.eq500} fails as in the case bounded case shown in Proposition~\ref{TSS.prop20}.

\prop\label{IFE.prop20}
For any $\rho \in (0, \frac 12)$, 
\[
\inf_{(x, y) \in \sd{(\BbR_+)^2}{\!\rm diag}}
 V(x, |x - y|)\phi(|x - y|)J^{(\infty)}(x, y, \rho{|x - y|}) = 0
\]
\endprop

Its proof is exactly the same as that of Proposition~\ref{TSS.prop20}.

\end{document}